\documentclass[a4paper]{amsart}
\pdfoutput=1

\usepackage[T1]{fontenc}
\usepackage{lmodern}
\usepackage[spacing,kerning, tracking]{microtype}

\usepackage[style=numams, sorting=nyt, isbn=false,  backref=true, natbib=true]{biblatex}
\usepackage{amsmath,amssymb, amscd}

\usepackage{enumitem}
\usepackage{verbatim}

\usepackage[pdfa]{hyperref}

\newcommand{\ichselber}{Frieder Ladisch}
\newcommand{\meineaddresse}{Universit\"{a}t Rostock,
                        Institut f\"{u}r Mathematik,
                        Ulmenstr.~69, Haus~3,
                        18057 Rostock,
                        Germany}
\newcommand{\meineemail}{frieder.ladisch@uni-rostock.de}

\newcommand{\titellang}{Character correspondences induced by magic representations}
\newcommand{\titellangsep}{Character correspondences \\ induced by magic representations}
\newcommand{\titelkurz}{Character correspondences}

\newcommand{\komma}{, }
\newcommand{\stichwoerter}{character theory of finite groups\komma
   character correspondences\komma
   Clifford theory\komma
   modular representations\komma
   rationality properties of characters\komma
   Schur indices\komma
   Glauberman correspondence}

\newcommand{\mscwert}{Primary 20C15, Secondary 20C20, 20C25}

\hypersetup{pdfauthor=\ichselber,
            pdftitle={\titellang},
            pdfkeywords={\stichwoerter ,
                        2010 Mathematics subject classification: \mscwert}
            }

\usepackage{mathbef}

\newcommand{\alcl}[1]{\overline{#1}}    

\begin{document}
\title[\titelkurz]{\titellangsep}
\author{\ichselber}
\address{\meineaddresse}
\email{\meineemail}
\subjclass[2010]{\mscwert}
%
\begin{abstract}
Let $G$ be a finite group, $K$ a normal subgroup of $G$ and
  $H\leq G$ such that
  $G=HK$, and set $L=H\cap K$.
  Suppose $\vartheta \in \operatorname{Irr} K$ and
  $\varphi\in \operatorname{Irr} L$, and $\varphi$
  occurs in $\vartheta_L$ with multiplicity $n>0$.
  A projective representation of degree $n$ on $H/L$ is defined in
  this situation; if this representation is ordinary, it yields a
  Morita equivalence between $\mathbb{C}G e_{\vartheta}$ and
  $\mathbb{C}H e_{\varphi}$, and thus a
  bijection between
  $\operatorname{Irr}(G\mid \vartheta)$ and
  $\operatorname{Irr}(H\mid \varphi)$.
  The behavior of fields of values and Schur indices under this
  bijection is described.
  A modular version of the main result is proved.
  We show that the theory applies if $n$ and
  $\lvert H/L \rvert$ are coprime.
  Finally, assume that $P\leq G$ is a $p$-group
  with $P\cap K=1$ and $PK$ normal in $G$, that
  $H=\operatorname{\mathbf{N}}_G(P)$, and that
  $\vartheta$ and $\varphi$ belong to
  blocks of $p$-defect zero which are Brauer correspondents with
  respect to the group $P$.
  Then every block of $\mathbb{F}_pG$ or $\mathbb{Q}_pG$ lying
  over $\vartheta$ is Morita-equivalent to its Brauer
  correspondent with respect to $P$.
  This strengthens a  result of
  Turull~[\emph{Above the {G}lauberman correspondence}, Advances in
  Math. \textbf{217} (2008), 2170--2205].

\end{abstract}
\maketitle
\section{Introduction}
  Let $G$ be a finite group, $K\nteq G$ and $H\leq G$ such that
  $G=HK$, and set $L=H\cap K$.
  Suppose $\theta\in \Irr K$ and $\phi\in \Irr L$ are
  invariant in $G$ respective $H$.
  We wish to compare the sets of characters
  $\Irr(G\mid \theta)$ and $\Irr(H\mid \phi)$.
  Questions of this type occur in many applications, and so this
  situation has already been studied
  by many authors~\citep{dade76, dade80, i73, i82, imana07, turull08c}.
  In
  particular, there are many results that construct,
  in special cases of the above situation,
  a bijection between
  $\Irr(G\mid \theta)$ and $\Irr(H\mid \phi)$.

  Here we will assume that $(\theta_L, \phi)= n > 0$, that is,
  $\phi$ occurs in the restriction of $\theta$ to $L$ with
  non-zero multiplicity $n$.
  In this situation, an $n$\nbd dimensional projective
  representation of $H/L$ arises naturally, as we will see.
  If this representation turns out to be projectively equivalent
  to an ordinary, ``honest'' representation, it yields a bijection
  between $\Irr(G\mid \theta)$ and $\Irr(H\mid \phi)$.
  This establishes a new, quite general approach for finding such
  a character correspondence.

  Let us give a somewhat more detailed outline of the idea.
  Let $e_{\theta}\in \compl K$ and $e_{\phi}\in \compl L$ be the
  central primitive idempotents associated with the characters
  $\theta $ and $\phi$
  and set $i= e_{\theta}e_{\phi}$.
  This is a nonzero idempotent in $\compl K$.
  Let $S= (i\compl K i)^L = \C_{i\compl K i}(L)$,
  the centralizer of $L$ in $i\compl K i $.
  It can be shown that $S\iso \mat_n(\compl)$, an
  $n \times n$\nbd matrix ring over $\compl$.
  The group $H/L$ acts on $S$ by conjugation,
  and for every $Lh\in H/L$ there is $\sigma(Lh)\in S^{*}$ such that
  $s^h = s^{\sigma(Lh)}$ for all $s\in S$.
  This defines a projective representation
  $\sigma\colon H/L \to S$.
  If $\sigma $ is multiplicative, we call it a
  ``magic'' representation.
  We will show that every magic representation $\sigma$ determines uniquely
  a bijection
  \[\iota=\iota(\sigma)\colon
     \Irr(G\mid \theta) \to
     \Irr(H\mid \phi)
  \]
  with good compatibility properties
  (see Theorem~\ref{t:corr} for the precise statement).
  A property particular to our situation is the
  following:
  Let $\psi$ be the character of the magic
  representation $\sigma$. Then for $\chi\in \Irr(G\mid \theta)$
  we have
  \[ \sum_{\xi \in \Irr(H\mid \phi)} (\chi_H, \xi) \xi
     = \psi \chi^{\iota}.
  \]
  Note that the left hand side is the part of $\chi_H$ that
  lies above $\phi$.

  As most readers probably know, the
  $G$\nbd invariant character
  $\theta$ determines a
  cohomology class of $G/K$,
  that is, an element of $ H^2(G/K, \compl^*)$.
  We denote this element by $[\theta]_{G/K}$.
  Similarly, $\phi$ determines a cohomology class
  $[\phi]_{H/L}\in H^2(H/L, \compl^*)$.
  As $G/K\iso H/L$ canonically, it is meaningful to compare
  $[\theta]_{G/K}$ and $[\phi]_{H/L}$.
  One can show that $[\theta]_{G/K}[\phi]_{H/L}^{-1}$
  is just the cohomology class belonging to
  the projective representation
   $\sigma$ described above.
  This also explains the existence of the character bijection,
  since it is known that $[\theta]_{G/K}$ determines
  $\Irr(G\mid \theta)$, in some sense~\citep[Chapter~11]{isaCTdov},
  and similarly for $\phi$.

  The results described so far are proved in
  Section~\ref{sec:magic} and~\ref{sec:corr}.
  In fact, we work with a field $\crp{F}$ containing the values of
  $\theta$ and $\phi$ instead
  of $\compl$.
  We show also that our correspondence has good rationality
  properties.

  In Section~\ref{sec:magiccrossed}, we go one step further and
  skip the assumption that $\theta$ and $\phi$ are invariant in
  $H$.
  This can not be handled simply  by the well known Clifford
  correspondence between $\Irr(G\mid \theta)$ and
  $\Irr(G_{\theta}\mid \theta)$ for the following reason:
  Suppose $\xi \in \Irr(G_{\theta}\mid \theta)$ and
  $\chi=\xi^G$ is its Clifford correspondent.
  Then $\rats(\chi) \subseteq \rats(\xi)$, but this may be a
  strict containment. In particular, we may have
  $\rats(\theta) \nsubseteq \rats(\chi)$.
  Of course  $\rats(\chi_K)\subseteq \rats(\chi)$,
  and $\rats(\chi_K)$ depends only on $\theta$ and the embedding of
  $K$ in $G$, but not on $\chi$ itself.
  Since we want results as general as possible
  about the behavior of fields of values and
  Schur indices under our character correspondences,
  we have to work over a field not necessarily containing
  the values of $\theta$ and $\phi$.
  This involves various technical difficulties. Some preliminary
  material is contained in Section~\ref{sec:semi-inv}.

  In Section~\ref{sec:coprim} we show that the results apply when
  $n=(\theta_L, \phi)$ and $\abs{H/L}$ are coprime.
  This slightly generalizes a result of Schmid~\citep{schmid88}.

  In Section~\ref{sec:induc} we prove a technical result that can
  be used for inductive proofs, and
  in Sections~\ref{sec:lifting} and~\ref{sec:reduc} we study
  relations with modular representation theory.

  Finally in Section~\ref{sec:aboveglaub} we assume that,
  additionally, we have a $p$\nbd subgroup $P\leq G$ such that
  $KP \nteq G$ and $P\cap K=1$, that
  $\theta$
  and $\phi$ have $p$-defect zero
  (in $K$ respectively in $L$) and that their blocks are Brauer
  correspondents with respect to $P$, and we let $H=\N_G(P)$.
  (If $K$ happens to be a $p'$\nbd group, then
   $\theta$ and $\phi$ are Glauberman correspondents.)
  We show that our theory applies in this situation, thus getting
  Morita equivalences between group algebras
  over the prime field with $p$ elements and over the field of
  $p$\nbd adic numbers (Corollaries~\ref{c:glaubbr}
  and~\ref{c:glaub}). In fact, we even get Morita equivalences
  between blocks over $\theta$ and $\phi$ which are Brauer
  correspondents with respect to $P$.
  This generalizes a recent result of
  Turull~\citep{turull08c}
  and provides an alternative proof thereof.

  The theory of magic representations also applies above
  fully ramified sections of a group, as studied by
  Isaacs~\citep{i73} and others. Indeed, much more information on
  the character of the magic representation is known in this case.
  In particular, in another paper~\cite{ladisch11pre}
  we use the methods of the present paper to show
  that a character correspondence described by Isaacs preserves
  Schur indices over the rational numbers.

  The results of this paper are part of my doctoral thesis~\citep{ladisch09diss},
  but some of them appear here in a more general form than
  in my thesis.

\section{Central forms and central simple
subalgebras}\label{sec:cssa}
We review some preliminary material in this section that we need
later.
Let $C$ be a ring. We write $\mat_n(C)$ to denote the ring of
$n\times n$\nbd matrices with entries in $C$.
Let $E_{ij}$ be the matrix such that the entry with index $(i,j)$
is $1_C$, and all the other entries are $0_C$.
The set $\{ E_{ij}\mid i,j=1, \dotsc, n\}$ has the following
properties:
\[
E_{ij}E_{kl}= \delta_{jk}E_{il}\quad \text{and}\quad
    \sum_{i=1}^n E_{ii} = 1.
 \]
Now let $A$ be any ring. Any subset
$E=\{E_{ij}\mid i,j=1,\dotsc, n\}$ of $A$ with these properties is called
a full set of matrix units in $A$.
If such a full set of matrix units exists, then
$A\iso\mat_n(C)$, where $C=\C_A(E)$~\citep[17.4-17.6]{lamMR}.
The isomorphism sends the $n\times n$\nbd matrix $(c_{ij})$ to
$\sum_{ij}c_{ij}E_{ij}$.

Now suppose that $A$ is an $R$\nbd algebra,
where $R$ is some commutative ring, and
$E\subseteq A$ is a full set of matrix units.
Let $S$ be the $R$\nbd subalgebra generated by $E$.
Then, by the result cited above, we have
$S\iso \mat_n(R)$.
It is also clear that
$\C_A(S)=\C_A(E)=: C$.
It follows that $A\iso S\tensor_{R} C$ canonically,
via $s\tensor c\mapsto sc$.

In the next few results, we will consider the following situation:
Let $\crp{F}$ be a field and
$S$ a central simple $\crp{F}$\nbd algebra.
Suppose $S$ is a subalgebra of the
$\crp{F}$\nbd algebra $A$ and set
$C=\C_A(S)$.
The following lemma is probably well known:
\begin{lemma}\label{l:central}
  $S \tensor_{\crp{F}} C\iso A$ canonically
  (via $s\tensor c \mapsto sc$).
\end{lemma}
\begin{proof}
  Define an algebra homomorphism
  $\kappa\colon S\tensor_{\crp{F}}C \to A$ by
  $(s\tensor c)^{\kappa}= sc$.
  We have seen above that
  $\kappa$ is an isomorphism if
  $S\iso \mat_n(\crp{F})$.
  In the general case, there is a field $\crp{E}\geq \crp{F}$ such
  that $S\tensor_{\crp{F}}\crp{E}\iso \mat_n(\crp{E})$.
  Then
  $\C_{A\tensor_{\crp{F}}\crp{E}} (S\tensor_{\crp{F}}\crp{E})
   = C \tensor_{\crp{F}}\crp{E}$,
  and
  $(S\tensor_{\crp{F}}C) \tensor_{\crp{F}} \crp{E}
   \iso
   (S \tensor_{\crp{F}}\crp{E}) \tensor_{\crp{E}}
   (C\tensor_{\crp{F}}\crp{E})$.
   It follows that
   $ \kappa \tensor 1 \colon
      (S \tensor_{ \crp{F} } C) \tensor_{ \crp{F} } \crp{E}
      \to A \tensor_{ \crp{F} } \crp{E} $
   is an isomorphism,
   and thus $\kappa$ is also an isomorphism.
\end{proof}
\begin{notat}
  Let $R$ be a commutative ring and $A$ an $R$\nbd algebra.
  We denote by $\ZF(A, R)$ the set of central
 $R$\nbd forms on $A$, that is the set of $R$\nbd linear maps
 $\tau \colon A \to R$ with $\tau(ab)= \tau(ba)$ for all $a,b\in A$.
\end{notat}
Note that an $R$\nbd linear map $\tau\colon A \to R$ is a central
form if and only if $[A,A]= \{ \sum_i (a_i b_i - b_i
a_i)\}\subseteq \ker \tau$.
Using this, one easily proves:
\begin{lemma}\label{l:zfatensorb}
  Let $A$ and $B$ be $R$-algebras. Then
  \[ (\sigma,\tau)\mapsto \sigma\tensor\tau\in \ZF(A\tensor_{R}B,R),
     \quad (\sigma\tensor\tau)(a\tensor b)=\sigma(a)\tau(b)\]
  defines a canonical isomorphism
  \[\ZF(A,R)\tensor_{R}\ZF(B,R)\iso \ZF(A\tensor_{R}B,R).\]
\end{lemma}
We return to the situation where  $S$ is a central simple
algebra over the field $\crp{F}$.
We denote the
reduced trace
of the central simple
$\crp{F}$\nbd algebra $S$ by
$\tr_{S/\crp{F}}$ or simply $\tr$, if no confusion can arise.
Remember that the reduced trace is computed as follows:
first choose a
splitting field $\crp{E}$ of $S$ and an isomorphism
$\eps \colon S\tensor_{\crp{F}}\crp{E} \to \mat_n(\crp{E})$,
then let $\tr_{S/\crp{F}}(s)$ be the
ordinary matrix trace of $(s\tensor 1)^{\eps}$.
Then indeed $\tr_{S/\crp{F}}(s)$ is a well-defined  element of
$\crp{F}$~\citep[Section~9a]{reinerMO}.
The kernel of $\tr_{S/\crp{F}}$ is exactly the subspace
$[S,S]$.
Thus $\ZF(S, \crp{F})= \crp{F}\cdot \tr_{S/\crp{F}}\iso \crp{F}$.

Now suppose $S\leq A$ and $C= \C_A(S)$ as above.
Combining these remarks with Lemma~\ref{l:central} and
Lemma~\ref{l:zfatensorb}, we get that
\[ \ZF(C, \crp{F})
   \iso \ZF(S,\crp{F}) \tensor_{\crp{F}} \ZF(C,\crp{F})
   \iso \ZF(S\tensor_{\crp{F}} C, \crp{F})
   \iso \ZF(A, \crp{F}).\]
Any central form $\chi\in \ZF(A, \crp{F})$
can be written as $\tr_{S/\crp{F}}\tensor \tau$ for some
$\tau\in \ZF(C,\crp{F})$.
The next lemma describes $\tau$ in terms of
$\chi$.
\begin{lemma}\label{l:cforms}
  Pick $s_0\in S$ with $\tr_{S/\crp{F}}(s_0)=1$.
  Consider
  \[\ZF(A,\crp{F})\ni \chi \mapsto \chi^{\eps}\in \ZF(C,\crp{F}),
    \quad
   \chi^{\eps} (c) = \chi(s_0 c)
  \quad \text{for $c\in C$}.\]
  Then $\eps$ is an isomorphism, is independent of the choice of $s_0$ and
  \[ \chi(sc) = \tr_{S/\crp{F}}(s) \chi^{\eps}(c)
    \quad\text{for all $s\in S$ and $c\in C$}.
    \]
\end{lemma}
\begin{proof}
  We have $[S,S]C =[SC,S]\subseteq [A,A]\subseteq \ker \chi$ for
  any $\chi\in \ZF(A, \crp{F})$.
  Since $\dim_{\crp{F}}(S/[S,S])=1$,
  it follows that
  for any $s\in S$ we have
  $s -\tr_{S/\crp{F}}(s)s_0\in [S,S]$.
  Thus
  \[\chi(sc)
     = \chi \big( (s-\tr_{S/\crp{F}}(s)s_0) c + \tr_{S/\crp{F}}(s)s_0c
            \big)
     = \tr_{S/\crp{F}}(s)\chi(s_0 c).\]
  This shows the last claim.
  If $\tau\in \ZF(C,\crp{F})$, then
  $\widehat{\tau}(sc)= \tr_{S/\crp{F}}(s)\tau(c)$ defines a
  central form $\widehat{\tau}$ on $A$, since
  $A\iso S\tensor_{\crp{F}}C$.
  It is clear now that
  $\tau\mapsto \widehat{\tau}$ is the inverse of $\eps$.
  Since the definition of $\widehat{\tau}$ is independent of
  $s_0$, the map $\eps$ is independent of $s_0$, too.
  The proof is finished.
\end{proof}
\begin{lemma}\label{l:chars}
   Assume that $S\iso \mat_n(\crp{F})$.
   Then
   $\chi^{\eps}$ affords a $C$\nbd module
   if and only if $\chi$ affords an
  $A$\nbd module.
\end{lemma}
\begin{proof}
  Identify $A$ with $S\tensor C$.
  Let $V$ be a simple $S$\nbd module.
  Then $\tr_{S/\crp{F}}= \tr_V$.
   Thus if $\chi^{\eps}$ is the
  character of the
  $C$\nbd module $M$,
  then $\chi$ is the
  character of the
  $A$\nbd module $V\tensor_{\crp{F}}M$.
  Conversely, suppose $\chi$ is the character of the
  $S\tensor_{\crp{F}}C$\nbd module $N$.
   Since $S$ is split, there is an idempotent $e\in S$ of trace
  $1$.
  Then the character of $Ne$ as
  $C$\nbd module is obviously
  $\chi^{\eps}$.
\end{proof}
\begin{remark*}
  If $S$ is possibly not split and $V$ a simple
  $S$\nbd module,
  then $\tr_V = m\tr_{S/\crp{F}}$ for some
  $m\in \nats$ (the Schur index of $S$).
  Thus if $\chi^{\eps}$ affords the
  $C$\nbd module $M$, then
  $m\chi$ affords the
  $A$\nbd module $V\tensor M$.
\end{remark*}
\begin{remark}
  The proofs of the last lemmas are easier when
  $S\iso\mat_n(\crp{F})$.
  Moreover, in this case
   Lemmas~\ref{l:central},~\ref{l:cforms} and~\ref{l:chars}
  remain true for commutative rings $\crp{F}$, not just fields.
\end{remark}

Now let $R$ be a commutative ring and $A$ an $R$\nbd algebra.
An idempotent $i\in A$ is called a
\emph{full idempotent} of $A$ if $AiA=A$.
It is well known that $A$ and $iAi$ are
Morita equivalent when $i$
is full~\citep[18.30]{lamMR}.
The following is also well known.
\begin{lemma}\label{l:fullidem}
  Let $A$ be an ${R}$\nbd algebra and $i$ a full idempotent of $A$.
  Set $C=iAi$.
  Then restriction to $C$ defines an isomorphism
  from $ \ZF(A,{R})$ onto $\ZF(C,{R})$,
  and $\tau_{|C}$ is the character of a $C$-module if and only
  if $\tau$ is the character of an $A$-module.
\end{lemma}
\begin{proof}
  Suppose $1_A = \sum_{k=1}^r x_k i y_k$ with $x_k, y_k \in A$.
  If $\xi\in \ZF(C,{R})$, then
  $\widehat{\xi}$ defined by
  $\widehat{\xi}(b)= \sum_{k=1}^r \xi(iy_k b x_k i)$
  is a central form, as a routine calculation shows.
  The map $\xi\mapsto \widehat{\xi}$ is the inverse
  of restriction.

  It is clear that if $\chi$ is the character of the $A$\nbd module $M$,
  then $\chi_{|C}$ is the character of $Mi$ as $C$\nbd module.
  Conversely if $\chi_{|C}$ is the character of the $C$-module $N$,
  then $\chi$ is the character of the $A$-module
  $N\tensor_{C} iA$.
\end{proof}

\section{Magic representations}\label{sec:magic}
Throughout this section
we  assume the following situation:
 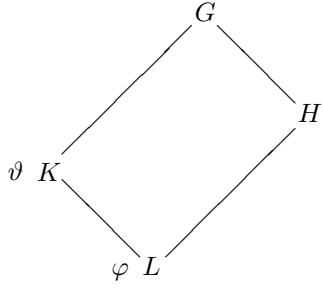
\begin{figure}[ht]
\setlength{\unitlength}{0.45ex}
\centering
\begin{picture}(70, 53)(-8,-1)
 \put(22.5,2.5){\line(1,1){25}}
\put(17.5,2.5){\line(-1,1){15}}
 \put(2.5,22.5){\line(1,1){25}}
\put(47.5,32.5){\line(-1,1){15}}
\put(18.0,-1){$L$}
\put(12,-1){$\phi$}
\put(-2.4,17.2){$K$}
\put(-8,17.4){$\theta$}
\put(28.1,48.4){$G$}
\put(48.0,28.7){$H$}
\end{picture}
\caption{The Basic Configuration}
\end{figure}
\begin{hyp}\label{h:bconf}
Suppose  $G= HK$ is a finite group where $K \nteq G$ and $H\leq G$,
 and set $L = H \cap K$.
 Let $\crp{F}$ be a field with algebraic closure
 $\alcl{\crp{F}}$.
 Let $\theta\in \Irr_{\alcl{\crp{F}}} K$ and
 $\phi\in \Irr_{\alcl{\crp{F}}} L $ be irreducible characters,
 such that
$\phi$ occurs in $\theta_L$ with  multiplicity $n>0$.
Assume that $\theta$ and $\phi$  are invariant in $H$
 and that     $\crp{F}(\phi)= \crp{F}(\theta)= \crp{F}$.
If $\crp{F}$ is a field of characteristic $p>0$,
assume that $\theta$ and $\phi$ belong to
$p$\nbd{}blocks of defect zero.
\end{hyp}
The assumption for charateristic $p$ assures that there
are central primitive idempotents
$e_{\phi}$ (respective $e_{\theta}$)
in $\crp{F}L$ (respective $\crp{F}K$) such that
$\phi(\crp{F}Le_{\phi})\neq 0$
(respective $\theta(\crp{F}Ke_{\theta})\neq 0$) and such that
$\crp{F}Le_{\phi}$ and $\crp{F}Ke_{\theta}$ are central simple
algebras over $\crp{F}$.
In characteristic zero, this is true anyway.
We set $i=e_{\phi}e_{\theta}$.
Observe that $i^h=i$ for all $h\in H$ and that
$i^2 =e_{\phi}e_{\theta}e_{\phi}e_{\theta}= e_{\phi}^2e_{\theta}^2=i$
since $e_{\theta}\in \Z(\crp{F}K)$.
We write $S=(i\crp{F}K i)^L =\C_{i\crp{F}K i}(L)$.
We fix this notation for the rest of this section.
\begin{lemma}\label{l:ctcentr}
   $S$ is a central simple $\crp{F}$\nbd algebra,
   its dimension over $\crp{F}$ is $n^2$, and
  $\C_{i\crp{F} K i} ( S )
   = \crp{F} L i \iso \crp{F} L e_{\phi}$.
\end{lemma}
\begin{proof}
 Let  $V$ be a $\alcl{\crp{F}} K$\nbd module affording $\theta$.
 Then
 $Vi = Ve_{\phi}$ is the $\phi$\nbd part of $V_{\alcl{\crp{F}} L}$, and thus
 $Vi \iso nU$ for some simple $\alcl{\crp{F}} L$\nbd module affording $\phi$.
 (In particular, $i\neq 0$.)
 The isomorphisms
 $\alcl{\crp{F}}K e_{\theta}
   \iso \enmo_{\alcl{\crp{F}}} V
   \iso \mat_{\theta(1)}(\alcl{\crp{F}})$
 restrict to isomorphisms
 $i\alcl{\crp{F}} K i \iso \enmo_{ \alcl{\crp{F}} }(Vi)
   \iso
 \mat_{n\phi(1)}(\alcl{\crp{F}})$.
 In particular,
 $\dim_{\crp{F}}(i\crp{F}K i) = n^2\phi(1)^2$.

 Since  $i\neq 0$,
 the ring homomorphism
 $\alpha \mapsto \alpha i = \alpha e_{\theta}$
 from $\crp{F} L e_{\phi}$ to $\crp{F} L i$ is not zero.
 As $\crp{F} L e_{\phi}$ is simple,
 the map $\alpha \mapsto \alpha i$ is injective. Thus
 $\crp{F} L e_{\phi} \iso \crp{F} L i$.

 The algebras $\crp{F} K e_{\theta}$ and $i\crp{F} K i$
 are central simple.
 By definition,
$S$ is just the centralizer of $\crp{F} L i$ in $i\crp{F} K i$. By the
 Centralizer Theorem~\citep[Theorem~3.15]{farbNA}, 
 $S$ is central simple, too, and
 the centralizer of $S$ is again $\crp{F} L i$. Also
 $i \crp{F} K i \iso S\tensor_{\crp{F}} \crp{F} L i$.
 From this  and from
 $\dim_{\crp{F}}(i\crp{F} K i) = n^2 \phi(1)^2$
 it follows that $\dim_{\crp{F}}S=n^2$ as claimed.
\end{proof}
\begin{cor}\label{c:ifgicentral}
  $i\crp{F}G i \iso S\tensor_{\crp{F}}C$, where
  $C= \C_{i\crp{F}G i}(S)$.
\end{cor}
\begin{proof}
  Clear by Lemma~\ref{l:central}.
\end{proof}
\begin{lemma}\label{l:fii}
  $\crp{F}Ke_{\theta}= \crp{F}K i \crp{F}K$ and
  $ \crp{F}G e_{\theta}= \crp{F}G i \crp{F}G$.
\end{lemma}
\begin{proof}
  The first assertion follows since $\crp{F}Ke_{\theta}$ is a
  simple ring and $i\neq 0$, and the second follows from the first.
\end{proof}
By the last two results and the results from
Section~\ref{sec:cssa}, the characters of
$\crp{F}Ke_{\theta}$ are in bijection with the characters of $C$.
We now work to find isomorphisms between $C$ and
$\crp{F}He_{\phi}$. Such isomorphisms exist under additional
conditions.

\begin{lemma}\label{l:projrep}
  Let $T$ be the inertia group of $\phi$ in $\N_G(L)$.
  Then $T/L$ acts on $S=(i \crp{F} K i)^L$ (by conjugation).
  There is  a projective representation
  $\sigma \colon T/L \to S$
  such that $s^g = s^{\sigma(Lg)}$ for all $s\in S$ and
  $g\in T$.
\end{lemma}
Observe that by Hypothesis~\ref{h:bconf}, we have $H\leq T$.
\begin{proof}
  As $\theta$ is $G$\nbd invariant, $T$ acts on $S$ and
  clearly $L$ acts trivially on $S$.
  Since $S$ is a central simple
  $\crp{F}$\nbd algebra, all automorphisms are inner by the
  Skolem-Noether Theorem
  \citep[Theorem~3.14]{farbNA}. 
  We can thus  choose $\sigma(Lg)\in S$ for every $g\in T$
  such that $s^g= s^{\sigma(Lg)}$ for all $s\in S $.
  This  determines $\sigma(Lg)$  up to
  multiplication with an element of $\Z(S)=\crp{F}i$.
  In this way we get a
  projective representation $\sigma$ from $T/L$ to $S$ with
  the desired property.
\end{proof}
We digress to prove a fact mentioned in the introduction.
We write $[\theta]_{ (G/K, \crp{F}) }$ for the cohomology class
in $H^2(G/K, \crp{F}^*)$ determined by $\theta$.
We write $[\theta]_{ (H/L, \crp{F}) }$ for its image in
$H^2(H/L, \crp{F}^*)$ under the isomorphism induced by the natural
isomorphism $G/K \iso H/L$.
\begin{remark}\label{r:cohomcl}
  Choose $\sigma$ as in Lemma~\ref{l:projrep}.
  Let $\alpha$ be the cocycle of $H/L$ defined by
  \[ \sigma(x)\sigma(y) = \alpha(x,y) \sigma(xy),
  \]
  and let $[\alpha]\in H^2(H/L, \crp{F}^*)$ be its cohomology class.
  Then
  \[  [\theta]_{ (H/L,\crp{F}) } = [\alpha] [\phi]_{ ( H/L,\crp{F}) }.
  \]
\end{remark}
\begin{proof}
  Let us first recall the definition of $[\theta]_{ (G/K,\crp{F}) }$ and
  $[\phi]_{ (H/L,\crp{F}) }$:
  For every $g\in G$ there exists $u_g\in \crp{F}Ke_{\theta}$ such
  that $a^g = a^{u_g}$ for all $a\in \crp{F}Ke_{\theta}$.
  This can be done such that $u_{gk}= u_g k$ and
  $u_{kg}= ku_g$ for all $k\in K$ and $g\in G$.
  Define $f\in Z^2(G, \crp{F}^*)$ by
  $u_x u_y = f(x,y)u_{xy}$.
  Then $f$ is constant on cosets of $K$ and thus may be viewed as
  an element of $Z^2(G/K, \crp{F}^*)$.
  Its cohomology class is, by definition, $[\theta]_{ (G/K,\crp{F}) }$.
  (A more usual approach
   is probably this~\citep[Theorem~11.2]{isaCTdov}: Suppose $\theta$ is afforded
   by a representation
   $D\colon K\to \mat_{\theta(1)}(\crp{F})$.
   (This may not be the case!)
   Extended $D$ to a projective representation of $G$.
   The corresponding factor set $f$ defines $[\theta]_{ (G/K,\crp{F})}$.
   Note that if $u_g$ as above are defined, we may extend
   $D$ by setting $\widehat{D}(g) = D(u_g)$.)

  Similarly, choose $v_h\in \crp{F}Le_{\phi}$ such that
  $b^h= b^{v_h}$ for all $b\in \crp{F}Le_{\phi}$.

  Since $i\crp{F}K i \iso S\tensor_{\crp{F}} \crp{F}Le_{\phi}$, we
  get that for
  $s\in S$ and $b\in \crp{F}Le_{\phi}$, we have
  \[ (sb)^{\sigma(h)v_h} = s^{\sigma(h)}b^{v_h}
       = s^h b^h = (sb)^h = (sb)^{u_h}.\]
  Since $i\crp{F}Ki$ is central simple, it follows
  that $iu_h = \lambda_h \sigma(h)v_h$ with $\lambda_h\in \crp{F}$.
  The proof follows.
\end{proof}
\begin{defi}\label{d:magic}
  In the situation of Hypothesis~\ref{h:bconf}, we say that
  \[\sigma \colon H/L \to S = (i\crp{F}Ki)^L
  \]
  is a magic representation
  (for $G, H, K, L, \theta, \phi$), if
  \begin{enums}
  \item $\sigma(Lh_1h_2)=\sigma(Lh_1)\sigma(Lh_2)$ for all $h_1$,
  $h_2\in H$ and
  \item $s^h = s^{\sigma(Lh)}$ for all $s\in S$ and $h\in H$.
  \end{enums}
  The character of a magic representation, that is the function
  $\psi \colon H/L \to \crp{F}$ defined by $\psi(Lh)= \tr_{S/\crp{F}}(\sigma(Lh))$, is called
  a magic character.
\end{defi}
\begin{thm}\label{t:C-iso}
  Assume Hypothesis~\ref{h:bconf} and let
  $\sigma \colon H/L \to S$ be a magic representation.
   Then the linear map
  \[ \kappa\colon \crp{F}H  \to C=\C_{i\crp{F}Gi}(S),\quad
     \text{defined by} \quad
     h \mapsto h \sigma(Lh)^{-1} \text{ for } h\in H ,\]
  is an algebra-homomorphism and induces an isomorphism
  $ \crp{F}H e_{\phi} \iso C$.
\end{thm}
\begin{proof}
  For $h\in H$ let $c_h = h \sigma(Lh)^{-1}= \sigma(Lh)^{-1}h$.
  (The inverse $\sigma(Lh)^{-1}$ is the inverse in $S$,
   so $\sigma(Lh)\sigma(Lh)^{-1}= i$.) Clearly $c_h\in C$.
  Note that
  \begin{align*}
     c_{g}c_{h}
      &= g \sigma(Lg)^{-1}h \sigma(Lh)^{-1}
         = gh (\sigma(Lg)^{-1})^{\sigma(Lh)}\sigma(Lh)^{-1}\\
      &= gh \sigma(Lh)^{-1}\sigma(Lg)^{-1}
         = gh \sigma(Lgh)^{-1} = c_{gh}.
  \end{align*}
  Thus extending the map $h \mapsto c_h$ linearly to $\crp{F}H$ defines an
  algebra homomorphism
  $ \kappa\colon \crp{F}H \to C$.
  For $l\in L $ we have
  $l \mapsto l \sigma(L)^{-1} = l i= l\cdot 1_C$.
  Thus $\kappa$
  restricted to $\crp{F}L$ is just multiplication with $i$,
  so that ${e}_{\phi}{\kappa}= e_{\phi}i= i= 1_C$,
  and any other central idempotent of
  $\crp{F}L$ maps to zero.
  Therefore
  \[(\crp{F}L){\kappa} = (\crp{F}Le_{\phi})\kappa = \crp{F} L i.
  \]

  For any $h\in H$ we have
  $(\crp{F}L e_{\phi} h){\kappa} = \crp{F} L i c_h$.
  Let $T$ be a transversal for the cosets of $L$ in $H$.
  As $\crp{F}He_{\phi}
      = \bigoplus_{t\in T} \crp{F} Le_{\phi} t$,
  the proof will be finished if we show that
  $ C = \bigoplus_{t\in T} \crp{F}L i c_t$.
  But this follows from standard arguments from the theory of
  group graded algebras:
  The decomposition $\crp{F} G = \bigoplus_{t\in T} \crp{F} Kt$
  yields the decomposition
  $i\crp{F} G i = \bigoplus_{t \in T} i\crp{F} Kt i $.
  For $Kg\in G/K$ set
  $C_{Kg}= C\cap \crp{F} Kg = \C_{i\crp{F} K g i}(S)$.
  Since $S \subseteq  i\crp{F} K i$, the centralizer $C$ of $S$
  inherits the grading of $i\crp{F}Gi$:
  \[C=\bigoplus_{Kg\in G/K} C_{Kg} = \bigoplus_{t\in T} C_{Kt}.\]
  As
  $c_t \in C \cap  i\crp{F} K i t = C\cap i \crp{F} K t i = C_{Kt}$ and
  $c_t$ is a unit of $C$,
  we conclude that
  \[ C_{Kt} = C_{Kt} c_t^{-1}c_t  \subseteq   C_Kc_t \subseteq C_{Kt},\]
  so equality holds throughout.
  As $C_{K} = \C_{i\crp{F}K i}(S) = \crp{F}Li$  by Lemma~\ref{l:ctcentr},
  the proof follows.
\end{proof}
\begin{remark*}
  If the projective representation of Lemma~\ref{l:projrep} is not
  equivalent to an ordinary representation, we still get some
  result. Let $\alpha$ be a factor set associated with $\sigma$,
  that is, $\sigma(x)\sigma(y)=\alpha(x,y)\sigma(xy)$ for $x,y\in
  H/L$.
  We may view $\alpha$ as
  a factor set of $H$ which is constant on cosets of $L$.
  Let $\crp{F}^{\alpha^{-1}}[H]$ denote the twisted group algebra
  with respect to $\alpha^{-1}$, that is,
  $\crp{F}^{\alpha^{-1}}[H]$ has a basis $\{b_h\mid h\in H\}$
  where $b_h b_g = \alpha(h,g)^{-1}b_{hg}$.
  The
  ordinary group algebra $\crp{F}L$ can be embedded in the
  twisted group algebra $\crp{F}^{\alpha^{-1}}[H]$ in an obvious way,
  and thus it is meaningful to view $e_{\phi}$ as an idempotent in
  $\crp{F}^{\alpha^{-1}}[H]$.
  With this notation, nearly the same proof as above shows that
  $C\iso \crp{F}^{\alpha^{-1}}[H]e_{\phi}$.
\end{remark*}
The following is an immediate consequence:
\begin{cor}\label{c:maincor1}
  Assume Hypothesis~\ref{h:bconf} and that there is a
  magic
  representation for this configuration.
  Then $i\crp{F}Gi \iso S\tensor_{\crp{F}} \crp{F}He_{\phi}$ and if $S\iso\mat_n(\crp{F})$,
  then $\crp{F}Ge_{\theta}$ and $\crp{F}He_{\phi}$ are Morita equivalent.
\end{cor}
\begin{proof}
  The first assertion follows by combining Theorem~\ref{t:C-iso}
  with Corollary~\ref{c:ifgicentral}. If $S\iso \mat_n(\crp{F})$, then
  $\crp{F}He_{\phi}$ and $i\crp{F}Gi\iso \mat_n(\crp{F}He_{\phi})$ are
  Morita-equivalent, and $\crp{F}Ge_{\theta}$ and $i\crp{F}Gi$ are
  Morita-equivalent,
  since $i$ is a full idempotent in $\crp{F}Ge_{\theta}$
  by Lemma~\ref{l:fii}.
\end{proof}

\section{Character correspondences}\label{sec:corr}
If $\crp{F}\leq \compl$, then every magic representation yields a
character correspondence between
$\Irr(G\mid \theta)$ and $\Irr(H\mid \phi)$.
We need some additional general notation to state the properties of this
correspondence.
\begin{notat}\label{not:brauerclass}\hfill
  \begin{enums}
  \item For any central simple $\crp{F}$-algebra $S$, let
        $[S]$ denote the algebra equivalence class of
        $S$ in the Brauer group of $\crp{F}$.
  \item For $\chi\in \Irr G$ and $\crp{F}\leq \compl$, set
        $[\chi]_{\crp{F}} = [\crp{F}(\chi)Ge_{\chi}]$.
        This is an element of the Brauer group of
        $\crp{F}(\chi)$.
  \end{enums}
\end{notat}
\begin{notat}
  For $\gamma \colon G \to \compl$ a class function and
  $\theta$ an irreducible character of some subgroup, let
$\gamma_{\theta}$ be the class function defined
 by
 \[ \gamma_{\theta}
     = \sum_{\chi \in \Irr(G\mid \theta)}
            (\gamma, \chi)_{G} \chi.
 \]
 This is the part of $\gamma$ that lies above $\theta$.
 We will need the case where $\theta$ is an irreducible character
 of some normal subgroup
 and  invariant in $G$.
 Note that then
 $\gamma_{\theta}(g)=\gamma(ge_{\theta})$.
\end{notat}
\begin{thm}\label{t:corr}
  Assume Hypothesis~\ref{h:bconf} with $\crp{F}\leq \compl$.
  Let
  $\sigma \colon H/L \to S^*$
  be a magic representation.
  Pick $s_0\in S$ with
  $\tr_{S/\crp{F}}(s_0)=1$.
  Let $U\leq G$ with $K\leq U$.
  For $\chi \in \compl[\Irr(U\mid \theta)]$, define
  \begin{equation}\label{equ:iota}
    \chi^{\iota}(h)= \chi(s_0 \sigma(Lh)^{-1}h).
  \end{equation}
  Then
  $\iota= \iota(\sigma)\colon
  \compl[\Irr(U\mid \theta)]
  \to \compl[\Irr(U\cap H\mid \phi)]$
  is an isomorphism independent of the choice of $s_0$.
  Let $\psi$ be the character of $\sigma$.
  The correspondence $\iota$ has
  the following properties:
  \begin{enums}
  \item \label{i:c_irr}
        $\chi \in \Irr(U\mid \theta)$ if and only if
        $\chi^{\iota} \in \Irr(U\cap H\mid \phi)$.
  \item \label{i:c_isometr}$(\chi_1^{\iota}, \chi_2^{\iota})_{U\cap H}
         = (\chi_1, \chi_2)_U$ for
         $\chi_1$, $\chi_2\in \compl[\Irr(U\mid \theta)]$.
  \item \label{i:c_prop}$\chi(1)/\theta(1) = \chi^{\iota}(1)/ \phi(1)$.
  \item \label{i:c_resuv}
        $(\chi_{U_1})^{\iota} = (\chi^{\iota})_{U_1\cap H}$
        for $K\leq U_1\leq U_2\leq G$ and
        $\chi\in \compl[\Irr(U_2\mid\theta)]$.
  \item \label{i:c_induv}$(\tau^{U_2})^{\iota}= (\tau^{\iota})^{U_2\cap H}$ for
        $K\leq U_1 \leq U_2 \leq G$ and
        $\tau \in \compl[\Irr(U_1\mid\theta)]$.
  \item \label{i:c_conj}$(\chi^{\iota})^h = (\chi^h)^{\iota}$ for $h\in H$.
  \item \label{i:c_lin}$(\beta\chi)^{\iota} = \beta \chi^{\iota}$
        for all $\beta \in \compl[\Irr(U/K)]$.
  \item \label{i:c_aut}If $\alpha$ is a field automorphism fixing $\crp{F}$, then
        $(\chi^{\alpha})^{\iota} = (\chi^{\iota})^{\alpha}$;
        and \label{i:c_field}
        $\crp{F}(\chi)= \crp{F}(\chi^{\iota})$.
  \item \label{i:c_brauer}
        $[\chi]_{\crp{F}} = [S\tensor_{\crp{F}}\crp{F}(\chi)][\chi^{\iota}]_{\crp{F}}$
        for $\chi\in \Irr(U\mid \theta)$.
  \item \label{i:c_res}$(\chi_{U\cap H})_{\phi}
          = \psi   \chi^{\iota}$ for
          $\chi\in \compl[\Irr(U\mid \theta)]$.
  \item \label{i:c_ind}$(\xi^U)_{\theta} = \overline{\psi}\xi^{\iota^{-1}}$ for
         $\xi \in \compl[\Irr(U\cap H\mid \phi)]$.
         (Here we view $\psi$ as character on $G/K$,
          via the canonical isomorphism $G/K\iso H/L$.)
  \end{enums}
\end{thm}
\begin{proof}
  Note that
  \[ \compl[\Irr(G\mid \theta)]\iso \ZF(\compl G e_{\theta},\compl)
     \quad \text{and} \quad
     \compl[\Irr(H\mid \phi)]\iso \ZF(\compl H e_{\phi},\compl)
  \]
  naturally.
  We work first over $\compl$.
  Set $S_{\compl}= (i\compl K i)^L$ and
  $C_{\compl}= \C_{i\compl G i}(S_{\compl})$.

  By Lemma~\ref{l:fii} we have
  $\compl G e_{\theta} = \compl G i \compl G$.
  Lemma~\ref{l:fullidem} yields that restriction defines
  an isomorphism
  $\ZF(\compl G e_{\theta}, \compl)
   \to \ZF(i\compl G i, \compl)$.

  Since
  $S_{\compl}= (i\compl K i)^L
    \iso \mat_n(\compl)$ by
  Lemma~\ref{l:ctcentr}, Lemma~\ref{l:cforms} yields an
  isomorphism
  $\eps\colon \ZF(i\compl G i, \compl)\to \ZF(C_{\compl},
  \compl)$.

  Finally, the isomorphism
  $\kappa\colon \compl H e_{\phi} \to C_{\compl}$ of
  Theorem~\ref{t:C-iso} (applied with $\crp{F}=\compl$)
  yields an isomorphism
  $\kappa^{*}\colon
   \ZF(C_{\compl}, \compl)
   \to \ZF(\compl H e_{\phi},\compl)$.

  We claim that $\iota$ is the composition of these three
  isomorphisms:
  \[ \ZF(\compl G e_{\theta},\compl)
     \xrightarrow[\text{(\ref{l:fullidem})}]{\text{Res}}
     \ZF(i\compl G i, \compl)
     \xrightarrow[\text{(\ref{l:cforms})}]{\eps}
     \ZF(C_{\compl}, \compl)
     \xrightarrow[\text{(\ref{t:C-iso})}]{\kappa^{*}}
     \ZF(\compl H e_{\phi}, \compl)
  \]
  To see this,
  let $h\in H$ and $\chi\in \compl[\Irr(G\mid \theta)]$.
  Then
  \[
    \chi^{\text{Res}\: \eps \kappa^{*}}(h)
      = \chi^{\eps}( \sigma(Lh)^{-1}h )
      = \chi(s_0 \sigma(Lh)^{-1}h).
  \]
  Moreover, Lemma~\ref{l:cforms} yields that $\eps$, and thus
  $\iota = \text{Res}\cdot \eps \kappa^{*}$, is independent of the
  choice of $s_0$.

  It follows from Lemma~\ref{l:fullidem} and Lemma~\ref{l:chars}
  that the first two isomorphisms send characters of (irreducible)
  modules to characters of (irreducible) modules.
  This is true for $\kappa^{*}$, too, since $\kappa$ is an
  isomorphism.
  Thus
  $\iota\colon \ZF(\compl G e_{\theta},\compl)
   \to \ZF(\compl H e_{\phi}, \compl)$
   is an isomorphism
   that sends irreducible
   characters to irreducible characters.
   It follows that $\iota$ respects the inner product on the space
   of class functions on $G$ respective $H$.

  Of course our reasoning so far applies to any subgroup $U$
  with $K\leq U\leq G$  instead of $G$, and to $V=H\cap U$
  instead of $H$.
  Thus we get an isometry
  $\iota\colon \ZF(\compl U e_{\theta},\compl)
   \to \ZF(\compl V e_{\phi}, \compl)$ for every such subgroup
   $U$.
   We use the same letter $\iota$ to denote all these
   isometries and their union.
   We have now established that $\iota$ is well defined and
   bijective, as well as
  Properties~\ref{i:c_irr} and ~\ref{i:c_isometr}
  of $\iota$.

  To show that \ref{i:c_prop} holds,
  we choose $s_0 = (1/n)i = (1/n)e_{\theta}e_{\phi}$.
  Remember that $n = (\theta_L, \phi)_L$,
  so that $\theta(e_{\phi})= n\phi(1)$.
  Thus
  \[ \frac{\chi^{\iota}(1)}{\phi(1)}
     = \frac{\chi((1/n)e_{\theta}e_{\phi})}{\phi(1)}
     = \frac{ \chi(e_{\phi}) }{n\phi(1)}
     = \frac{ (\chi_K, \theta) \theta( e_{\phi})}{ n\phi(1)}
     = (\chi_K, \theta)
     = \frac{\chi(1)}{\theta(1)}.
      \]
  Properties~\ref{i:c_resuv}--\ref{i:c_aut}
  are consequences of a general theorem about graded
  Morita equivalences~\citep[Theorem~3.4]{marcus08},
  which applies here.
  It is, however, not difficult to prove them directly. We do this
  for~\ref{i:c_lin} and~\ref{i:c_aut}.

  Let $\beta \in \compl[\Irr(G/K)]$.
  Then
  \begin{align*}
    (\beta\chi)^{\iota}(h) &=  (\beta\chi)(s_0\sigma(h)^{-1}h).
  \intertext{Writing %
             $s_0\sigma(h)^{-1} = \sum_{k\in K}\lambda_k k$ %
               with $\lambda_k\in \crp{F}$, we get}
    (\beta\chi)^{\iota}(h) &= \sum_{k\in K} \lambda_k \beta(kh)\chi(kh)
                          = \beta(h) \sum_{k\in K} \lambda_k \chi(kh)\\
                         & = \beta(h) \chi(s_0\sigma(h)^{-1}h)
                          = \beta(h) \chi^{\iota}(h).
  \end{align*}
  This proves \ref{i:c_lin}.

  To prove \ref{i:c_aut},
  write $s_0\sigma(h)^{-1} = \sum_{k\in K}\lambda_k k$
  as above.
  Since $\lambda_k\in \crp{F}$, it follows
  \begin{align*}
    (\chi^{\alpha})^{\iota}(h)
      &= \chi^{\alpha}(s_0\sigma(Lh)^{-1}h )
       = \sum_{k\in K}\lambda_k \chi^{\alpha}(kh)
         \\
      &= \sum_{k\in K} \lambda_k \chi(kh)^{\alpha}
        = \left( \sum_{k\in K} \lambda_k \chi(kh)
         \right)^{\alpha}
        \\
      &= (\chi^{\iota}(h))^{\alpha}.
  \end{align*}
  Thus $(\chi^{\alpha})^{\iota} = (\chi^{\iota})^{\alpha}$,
  and $\crp{F}(\chi)= \crp{F}(\chi^{\iota})$ follows from this.

  Now let $\chi \in \Irr(G\mid \theta)$.
  Then $[\chi]_{\crp{F}}$ is the equivalence class of
  $\crp{F}(\chi)Ge_{\chi}$ in the Brauer group of $\crp{F}(\chi)$.
  By \ref{i:c_aut} we may assume that
  $\crp{F}=\crp{F}(\chi)=\crp{F}(\chi^{\iota})$.
  The isomorphism of Corollary~\ref{c:maincor1} sends
  $S\tensor_{\crp{F}} \crp{F}H e_{\chi^{\iota}}$ onto
  $i\crp{F}Gi e_{\chi} = i\crp{F}Ge_{\chi}i$.
  But this central simple algebra is in the equivalence class of
  $\crp{F}Ge_{\chi}$, since $ie_{\chi}$ is an idempotent.
  This proves~\ref{i:c_brauer}.

  Property~\ref{i:c_res}:
  For $\chi \in \compl[\Irr(G\mid \theta)]$ we have
  \[(\chi_H)_{\phi}(h) = \chi (he_{\phi})
    = \chi(he_{\phi}e_{\theta}) = \chi(hi).\]
  Now note that
  $hi= ih = \sigma(Lh) \sigma(Lh)^{-1} h $ for $h\in H$.
  By Lemma~\ref{l:cforms} and the definition of $\iota$, we have
  \[\chi(hi)
       = \tr_{S/\crp{F}}(\sigma(Lh)) \chi(s_0\sigma(Lh)^{-1}h)
          = \psi(h) \chi^{\iota}(h)\]
  as claimed in \ref{i:c_res}.

  To prove \ref{i:c_ind}, it suffices to show that
  $( \overline{\psi} \xi^{\iota^{-1}} , \chi )_G= (\xi^G, \chi)_G$ for all
  $\chi \in \Irr(G\mid \theta)$.
  Using what we have already proved, we get
  \begin{alignat*}{2}
    (\overline{\psi}\xi^{\iota^{-1}}, \chi)_G
      &= ( \xi^{\iota^{-1}}, \psi \chi)_G  && \\
      &= ( \xi  , (\psi\chi)^{\iota})_H
             & & \quad \text{(as $\iota$ is an isometry)} \\
      &= ( \xi, \psi \chi^{\iota})_H
         && \quad \text{(by \ref{i:c_lin}.)}  \\
      &= ( \xi, (\chi_H)_{\phi} )_H
         && \quad \text{(by \ref{i:c_res}.)} \\
      &= (\xi, \chi_H )_H
         &&\quad\text{(as $\xi\in \compl[\Irr(H\mid \phi)]$)} \\
      &= ( \xi^G, \chi )_G
  \end{alignat*}
  as was to be shown.
  The proof is complete.
\end{proof}
The bijection of the theorem  depends on the
magic\index{Magic representation} representation $\sigma$.
If such a representation exists,
it is unique up to multiplication
with a linear character of $H/L$ (with values in $\crp{F}$).
Different choices of $\sigma$ give bijections which differ by
multiplication with a linear character of $H/L$.
Note that if $\psi(h)\neq 0$ for all $h\in H$, then
$\chi^{\iota}$ is determined by the equation
$(\chi_H)_{\phi}= \psi \chi^{\iota}$.
Otherwise one needs the representation $\sigma$ to compute
$\chi^{\iota}$ from Equation~\eqref{equ:iota}.

To formulate the next result correctly,
we need the
\emph{reduced norm} of a central simple algebra $S$
over $\crp{F}$, which we denote by $\nr_{S/\crp{F}}$ or simply
$\nr$, if $S$ and $\crp{F}$ are clear from context.
Remember that it is defined as follows:
First, choose a splitting field
$\crp{E}\geq \crp{F}$ of $S$ and an isomorphism
$\eps\colon S\tensor_{\crp{F}}\crp{E}\iso \mat_n(\crp{E})$.
Then for $s\in S$
define $\nr_{S/\crp{F}}(s) = \det( \eps(s\tensor 1))$, where
$\det$ denotes the determinant of the matrix ring
$\mat_n(\crp{E})$.
It can be shown that $\nr(s)$ is independent of the particular
isomorphism $\eps$ and of the choice of $\crp{E}$, and that
$\nr(s)\in \crp{F}$~\citep[{\S}~9a]{reinerMO}.
If $\sigma\colon H/L\to S^*$ is a magic representation, then
$x\mapsto \nr(\sigma(x))$ defines a linear character,
which we denote simply by $\det \sigma$.
\begin{remark}\label{r:corrdet}
  Let $\pi$ be the set of  prime divisors of $n$.
  If there is any magic\index{Magic representation} representation,
  then there is a magic representation $\sigma$ such that
  $\det \sigma$ has order a $\pi$\nbd number.
\end{remark}
\begin{proof}
  Suppose $\sigma \colon H/L \to S$ is given.
  Let $\lambda= \det \sigma$, a linear character of $H/L$.
  Let $b$ be the $\pi'$\nbd part of $\ord(\lambda)$.
  As $n = \dim \sigma$ is $\pi$, there is $r\in \ints$ with
  $rn+1 \equiv 0 \mod b$.
  Then $\det( \lambda^r \sigma) = \lambda^{rn+1}$ has $\pi$\nbd order.
\end{proof}
Let $\alpha\in \Aut\crp{F}$ be a field automorphism. Then $\alpha$
extends naturally to an automorphism of the group algebra
$\crp{F}G$, which we denote also by $\alpha$.
Remember that $\Aut\crp{F}$ acts on the set
of class functions $\chi\colon G\to \crp{F}$ by
$\chi^{\alpha}(g)=\chi(g)^{\alpha}$ ($g\in G$).
As usual, a class function extends linearly to a funcion
$\crp{F}G\to \crp{F}$.
Note that then
$\chi^{\alpha}(c^{\alpha})=\chi(c)^{\alpha}$ for
$c =\sum_{g}c_g g\in \crp{F}G$ arbitrary.
This will be used in the proof of the next proposition.
\begin{prop}\label{p:fieldisos}
  Let
  $\mathcal{B}=(G,H,K,L,\theta, \phi)$
  be a configuration such that Hypothesis~\ref{h:bconf}
  holds
  over the field $\crp{F}$,
  and let $\alpha\in \Aut \crp{F}$.
  If $\sigma\colon H/L \to S$ is a magic representation for
  $\mathcal{B}$ with magic character $\psi$, then
  \[\sigma^{\alpha}\colon H/L\to S
  \quad \text{defined by}\quad
  \sigma^{\alpha}(h) = \sigma(h)^{\alpha}\]
  is a magic representation for
  $\mathcal{B}^{\alpha}=(G,H,K,L,\theta^{\alpha},\phi^{\alpha})$
  with magic character $\psi^{\alpha}$.
  Let $\iota(\sigma)$ and $\iota(\sigma^{\alpha})$ be the
  associated character correspondences.
  Then
  \[ \left( \chi^{\alpha} \right)^{ \iota(\sigma^{\alpha}) }
      = \left( \chi^{\iota(\sigma)} \right)^{ \alpha }
      \quad\text{for} \quad \chi\in \Irr(G\mid
  \theta).\]
\end{prop}
\begin{proof}
  Note that $e_{\theta}^{\alpha} = e_{\theta^{\alpha}}$ for the
  central primitive idempotent belonging to $\theta$.
  Thus $\sigma^{\alpha}\colon H/L\to S^{\alpha}
        = (i^{\alpha}\crp{F}^{\alpha}K i^{\alpha})^{L}$
  with $i^{\alpha}= e_{\theta^{\alpha}}e_{\phi^{\alpha}}$
  is a magic representation for the
  configuration $\mathcal{B}^{\alpha}$.

  The isomorphism $\alpha$ maps the reduced trace of $S$ to the reduced trace
  of $S^{\alpha}$, by uniqueness of the reduced trace, and so
  $\psi^{\alpha}$ is the character of $\sigma^{\alpha}$.

  Let $\chi\in \Irr(G\mid \theta)$ and pick
  $s_0\in S$ with reduced trace $1$.
  Thus $\tr_{S^{\alpha}/\crp{F}} (s_0^{\alpha}) =1$.
  Therefore
 \begin{align*}
  ( \chi^{\alpha} )^{\iota(\sigma^{\alpha} )} ( h )
    &= \chi^{\alpha} ( s_0^{\alpha}
                      \sigma^{\alpha} ( L h )^{-1}
                      h )
     = \chi^{\alpha}\left( (s_0 \sigma(Lh)^{-1}h)^{\alpha}\right)
        \\
    &= \chi(s_0 \sigma(Lh)^{-1}h)^{\alpha}
     = \chi^{\iota(\sigma)}( h)^{\alpha}
     = \left( \chi^{\iota(\sigma)} \right)^{ \alpha }(h) ,
 \end{align*}
 as was to be shown.
\end{proof}
\begin{prop}\label{p:centcorr}
Assume Hypothesis~\ref{h:bconf} and let $C\leq\C_H(S)$
with $L\leq C$.
\begin{enums}
\item \label{i:centcorrbij}For every $\chi \in \Irr( KC\mid \theta )$, there is a
      unique
      $\xi \in \Irr( C\mid \phi )$ such that $(\chi_C, \xi)_C>0$.
      This  defines a bijection between
      $\Irr(KC\mid \theta)$ and $\Irr( C\mid \phi )$,
      which has all the properties of Theorem~\ref{t:corr}
      with $\psi= n1_{C/L}$.
      This bijection is invariant under all automorphisms of $G$
      fixing $K$, $L$, $\theta$ and $\phi$.
\item \label{i:centcorr}Assume that $C\nteq H$, that $\xi\in \Irr(C\mid \phi)$
      is $H$\nbd stable and
      $\chi $ corresponds to $ \xi$.
      Then $T=(e_{\xi}\crp{F}KCe_{\chi}e_{\xi})^C$ and
      $S=(e_{\phi}\crp{F}Ke_{\theta}e_{\phi})^L$ are isomorphic as
      $H/C$\nbd algebras.
\end{enums}
\end{prop}
  Note that if $n=1$, then $\C_H(S)=H$.
  This case of Part~\ref{i:centcorrbij} is known
  and can be proved just using
  the orthogonality relations
  of ordinary character theory~\citep[Lemma~4.1]{i84}.
\begin{proof}
 Theorem~\ref{t:corr} applies to the configuration
 $(KC, C, K, L, \theta, \phi)$ with
 $\sigma\colon C/L \to S$, $\sigma(c)=1_S= i$ for all $c\in C$.
 Observe that then $\psi = n 1_{C}$.
 From Property~\ref{i:c_res} in Theorem~\ref{t:corr} it now follows that
 the restriction $\chi_C$ of
 every $\chi \in \Irr(KC\mid \theta)$ has a unique constituent in
 $\Irr(C\mid \phi)$, which occurs with multiplicity $n$, as
 claimed. Conversely, for $\xi \in \Irr(C\mid \phi)$, the induced
 character $\xi^G$ has a unique constituent lying in
 $\Irr(KC\mid \theta)$, by Property~\ref{i:c_ind}.
 The desired bijection
  is thus just the
 correspondence $\iota$ of Theorem~\ref{t:corr}.
 Part~\ref{i:centcorrbij} follows.

 Let $j= e_{\chi}e_{\xi}$, where we assume that
 $\xi$ and $ \chi$ are $H$\nbd invariant.
 Thus $T= (j\crp{F}KCj)^C$.
 The idempotent $j$ centralizes $S$, as $e_{\chi}$ is in the
 center of $\crp{F}KC$, and  $e_{\xi}\in \crp{F} C$
 with $C\leq \C_H(S)$.
 It follows that for every $s\in S$,
 we have $sj=js=jsj\in T$.
 As $e_{\chi}e_{\theta}= e_{\chi}$ and
 $e_{\xi}e_{\phi}= e_{\xi}$, it follows that
 $ij=j$.
 The map $s\mapsto sj$ is thus an algebra homomorphism
 from $S$ into $T$. Since $S$ is simple and
 $\dim_\crp{F} S = n =\dim_\crp{F} T$, the map is an isomorphism.
 It is compatible with the action of $H$ as $j$ is $H$\nbd stable.
\end{proof}
\begin{remark}
  Part~\ref{i:centcorrbij} says that
  $\abs{\Irr(KC\mid \theta)\cap \Irr(KC\mid \xi)}=1$
  for all $\xi\in \Irr(C\mid \phi)$.
  This contains the well known description of
  the characters of central products. Namely, let $G$ be the
  central product of $K$ and $C$, set $L=K\cap C$
  (so $L\leq \Z(G)$),
  and suppose $\phi\in \Irr L$.
  Then Part~\ref{i:centcorrbij} applies with $H=C$
  for every $\theta\in \Irr(K\mid \phi)$,
  since $C$ centralizes $K$.
  We get a bijection between
  $\Irr(KC\mid \phi)$ and
  $\Irr(K\mid \phi)\times \Irr(C\mid \phi)$.
\end{remark}
Part~\ref{i:centcorr} has an interesting consequence,
a ``going down'' result:
\begin{cor}\label{c:centcorrdown}
  In the situation of Part~\ref{i:centcorr} of
  Proposition~\ref{p:centcorr},
  assume that there exists a magic representation for
  the configuration
  $G, H, KC, C, \chi, \xi$.
  Suppose $L\leq D\leq C$, let
  $\widetilde{\chi}\in \Irr(KD\mid \theta)$ and
  $\widetilde{\xi}\in \Irr(LD\mid \phi)$ and assume
  $(\widetilde{\chi}, \widetilde{\xi})_{D}>0$.
  Suppose $D\nteq H$ and $\widetilde{\xi}$ is invariant in $H$.
  Then there is a magic representation $\sigma$
  for the configuration
  $(G,H,KD,D,\widetilde{\chi}, \widetilde{\xi})$
  with $C/D\leq \ker \sigma$.
\end{cor}
\begin{figure}[ht]
\setlength{\unitlength}{0.45ex}
\centering
  \begin{picture}(90,90)(-8,-5)
    \put(2.5,32.5){\line(1,1){10}}    
    \put(17.5, 47.5){\line(1,1){10}}  
    \put(32.5, 62.5){\line(1,1){15}}  
    \put(32.5,2.5){\line(1,1){10}}    
    \put(47.5, 17.5){\line(1,1){10}}  
    \put(62.5, 32.5){\line(1,1){15}}  
    \put(2.5, 27.5){\line(1,-1){25}}  
    \put(17.5,42.5){\line(1,-1){25}}  
    \put(32.5,57.5){\line(1,-1){25}}  
    \put(52.5, 77.5){\line(1,-1){25}} 
    \put(-2, 28){$K$} \put(-7,27){$\theta$}
    \put(11.1, 43){$KD$} \put(7, 47){$\widetilde{\chi}$}
    \put(26.1, 58){$KC$} \put(22, 62){$\chi$}
    \put(48.2, 78){$G$}
    \put(28, -1.8){$L$} \put(23, -3.5){$\phi$}
    \put(43,13){$D$}  \put(48,8){$\widetilde{\xi}$}
    \put(58, 28){$C$} \put(62,23){$\xi$}
    \put(78,48){$H$}
  \end{picture}
  \caption{Corollary~\ref{c:centcorrdown}}
\end{figure}
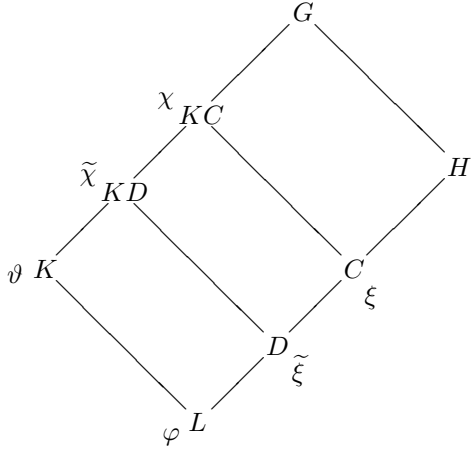
\begin{proof}
  Clear since
  \[( e_{\widetilde{\xi}} \crp{F} KD e_{\widetilde{\chi}} e_{\widetilde{\xi}}
   )^{D}
   \iso
   S
   \iso
   (e_{\xi}\crp{F} KC e_{\xi}e_{\chi} )^{C}
  \] as $H/C$-algebras for all such configurations.
\end{proof}
In terms of cohomology classes and in view of Remark~\ref{r:cohomcl},
this means
that if $[\chi]_{ (H/C,\crp{F}) }=[\xi]_{ (H/C,\crp{F}) }$, then
$[\widetilde{\chi}]_{ (H/D,\crp{F}) }
 =[\widetilde{\xi}]_{ (H/D,\crp{F}) }$.
This is related to some results
obtained in~\citep{imana07}.

Note that if  the configuration $(G,H,K,L,\theta,\phi)$
admits a magic representation $\sigma$ such that
$C/L\leq \ker \sigma$, then $\sigma$ may be viewed as a magic
representation for the configuration
$(G,H, KC, C, \chi, \xi)$.
It may be possible, however, that there are magic representations
for the configuration of $\theta$ and $\phi$, but no
magic representation whose
kernel
contains $C/L$.

\section{Semi-invariant characters}\label{sec:semi-inv}
Throughout this section,
let $\crp{F}$ be a field
with algebraic closure $\alcl{\crp{F}}$.
Let $K\nteq G$ and
$\theta\in \Irr_{\alcl{\crp{F}}}(K)$.
If $p=\cha \crp{F}>0$, then we assume that $\theta$ has
$p$\nbd defect zero.
There is a unique central primitive idempotent
$e_{\theta}$
of $\alcl{\crp{F}}K$, such that $\theta$
does not vanish on $\alcl{\crp{F}}K e_{\theta}$.
The assumption assures that
$\alcl{\crp{F}}K e_{\theta}\iso \mat_{d}(\alcl{\crp{F}})$
where $d$ is the dimension of a module affording $\theta$.

If $\alpha$ is an automorphism of a field $\crp{E}$, then
we denote also by $\alpha$ its natural extension to the group algebra
$\crp{E}G$, where $\alpha$ centralizes $G$.
We need the following well known fact.
\begin{lemma}\label{l:Traceidempotent}
  \[ e = \T_{\crp{F}}^{\crp{F}(\theta)}(e_{\theta})
    := \sum_{\alpha\in \Gal(\crp{F}(\theta)/\crp{F}) }
        (e_{\theta})^{\alpha}\]
  is
  the unique central primitive idempotent
  of\/ $\crp{F}K$
  for which $\theta(\crp{F}K e)\neq 0$.
\end{lemma}
\begin{proof}
Let $\Gamma= \Gal(\crp{F}(\theta)/\crp{F})$.
  Note that
  $e_{\theta^{\alpha}}= (e_{\theta})^{\alpha}$
  for $\alpha\in \Gamma$.
  Obviously, only the identity of $\Gamma$ fixes $\theta$ or $e_{\theta}$.
  Thus for $\alpha\neq \beta\in \Gamma$, we have
  $(e_{\theta})^{\alpha}(e_{\theta})^{\beta}=0$, so $e$ is an
  idempotent.

  It is clear that
  $e\in \Z(\crp{F}K)$ and that
  $\theta$ does not vanish on $\crp{F}Ke$.
  (Otherwise $\theta$ would vanish on $\alcl{\crp{F}}Ke_{\theta}$.)
  We claim that
  $e$ is a primitive idempotent in $\Z(\crp{F}K)$.
  Indeed, if $0\neq f\in \Z(\crp{F}K)$ is an idempotent with
  $fe=f$, then $e_{\theta}^{\alpha}f= e_{\theta}^{\alpha}$ for some
  $\alpha\in \Gamma$.
  Then for any $\sigma \in \Gamma$, we have
  $(e_{\theta})^{\alpha\sigma}f =
   (e_{\theta})^{\alpha\sigma}f^{\sigma}
   = (e_{\theta}^{\alpha} f)^{\sigma}
   = (e_{\theta}^{\alpha})^{\sigma}$, and thus $f=e$.
\end{proof}
\begin{notat}
 We write $e_{ (\theta, \crp{F}) }$ for the idempotent of
 Lemma~\ref{l:Traceidempotent}. In particular, if
 $\crp{F}=\crp{F}(\theta)$, then
 $e_{ (\theta, \crp{F}) }= e_{\theta}$.
\end{notat}
\begin{lemma}\label{l:fe_iso}
  \[  \crp{F}G_{\theta}e_{ (\theta, \crp{F}) } \ni a \mapsto
      ae_{\theta} \in \crp{F}(\theta)G_{\theta}e_{\theta} \]
  is an isomorphism of\/ $\crp{F}$\nbd algebras.
\end{lemma}
\begin{proof}
  Since $e_{\theta}\in \Z(\crp{F}(\theta)G_{\theta})$, the map is
  multiplicative.

  The inverse is given by the field trace
  $T=\T^{\crp{F}(\theta)}_{\crp{F}}$,
  extended from
  $\crp{F}(\theta)$ to
  $\crp{F}(\theta)G$ and
  then restricted to
  $\crp{F}(\theta)G_{\theta}e_{\theta}$:
  If $b\in \crp{F}(\theta)G_{\theta}e_{\theta}$, then
  $b^{\alpha}\in \crp{F}(\theta)G_{\theta}e_{\theta^{\alpha}}$ for
  $\alpha\in \Gal(\crp{F}(\theta)/\crp{F})$.
  If $\alpha\neq 1$, then $b^{\alpha} e_{\theta}=0$.
  Thus $T(b)e_{\theta} = be_{\theta}=b$.
  Conversely, for $a\in \crp{F}G_{\theta}e_{(\theta,\crp{F})}$ we have
  $T(ae_{\theta})= aT(e_{\theta})=ae_{(\theta,\crp{F})}=a$, since
  $T(e_{\theta})=e_{ (\theta, \crp{F}) }$.
\end{proof}
As a consequence, $\Z(\crp{F}Ke_{ (\theta, \crp{F}) })\iso \crp{F}(\theta)$.
This is of course well known.
An isomorphism is given by the central character $\omega_{\theta}$.

The following notation will be convenient:
 Let $K \nteq G$ and $e \in \Z(\crp{F}K)$.
 Let $G_e= \{g\in G \mid e^g =e\}$ and write $e^G$ for the
 idempotent defined by
 $e^G:= \T_{G_e}^G(e) = \sum_{g\in [G:G_e]}e^g$.

The following result is also well known:
\begin{prop}\label{p:cliffordcorr}
  Set  $e= e_{ (\theta, \crp{F}) }$ and $f= e^G$
  and let $T=G_e$ be the inertia group of $e$.
  Then\/ $\crp{F}Gf \iso \mat_{\abs{G:T}}(\crp{F}Te)$.
  Induction defines a bijection between\/
  $\Irr(T\mid \theta)$ and\/ $\Irr(G\mid \theta)$ that preserves field
  of values and Schur indices over $\crp{F}$.
  More precisely,
    $\crp{F}(\tau^G)=\crp{F}(\tau)$ and\/
  $[\tau^G]_{\crp{F}}=[\tau]_{\crp{F}}$
  for $\tau\in \Irr(T\mid \theta)$.
\end{prop}
(Here $[\tau]_{\crp{F}}$ is the equivalence class in the Brauer group
 associated with $\tau$, see~\ref{not:brauerclass}.)
\begin{proof}
  Let $G=\bigdcup_{u\in R} T u$.
  Since $e^g \neq e$ for $g\in G\setminus T$, it follows
  $e^g e = 0$ and thus $ege= 0$.
  From this it follows easily that
  $\{ E_{u,v}= u^{-1}e v \mid u, v \in R\}$ is a full set of
  matrix units in
  $\crp{F}G f$ and that
  $e\crp{F}G e = \crp{F}T e$.
  It is then routine to verify that
  \begin{align*}
    \crp{F}G f \ni a &\mapsto (e ua v^{-1} e)_{u,v\in R}
                                \in \mat_{\abs{G:T}}(\crp{F}T e)\\
    \intertext{and}
    \mat_{\abs{G:T}}(\crp{F}Te) \ni (b_{u,v})_{u,v}
                     &\mapsto \sum_{u,v\in R} u^{-1} b_{u,v}v
                              \in \crp{F}Ge\crp{F}G =\crp{F}G f
  \end{align*}
  are inverse isomorphisms.

  This isomorphism yields a bijection between isomorphism classes
  of $\crp{F}Gf$\nbd modules and
  $\crp{F}Te$\nbd modules.
  Let $V$ be an $\crp{F}Te$\nbd module.
  Then $V^{\abs{G:T}}$ is a module
  over $\mat_{\abs{G:T}}(\crp{F}Te)\iso \crp{F}Gf$,
  and $V^{\abs{G:T}}$ as $\crp{F}Gf$\nbd module is isomorphic
  to $V\tensor_{\crp{F}T}\crp{F}G$ via
  the map $(v_u)_{u\in R} \mapsto \sum_{u\in R} v_u \tensor u$.
  Thus induction yields a bijection between isomorphism classes of
  $\crp{F}Te$\nbd modules and $\crp{F}Gf$\nbd modules.
  (Alternatively, one can check directly that
   $(V\tensor_{\crp{F}T} \crp{F}G)e \iso V$ as
   $\crp{F}T$\nbd modules, and that, if $W$ is
   an $\crp{F}Gf$\nbd module, then
   $We\tensor_{\crp{F}T}\crp{F}G\iso W$ as
   $\crp{F}G$\nbd modules.)

  Applying the above reasoning over $\compl$ instead of $\crp{F}$
  yields that induction is a bijection between
  $\Irr(T\mid \theta)$ and $\Irr(G\mid \theta)$
  (this is the well known Clifford correspondence, anyway).
  That it preserves fields of values and
  Brauer equivalence classes can now be
  seen as follows: Suppose $\tau\in\Irr(T\mid \theta)$.
  Let $V$ be a simple $\crp{F}Te$\nbd module whose character
  contains $\tau$ as constituent.
  Then $\enmo_{\crp{F}T} V$ is a division ring in
  $[\tau]_{\crp{F}}$, and
  $\crp{F}(\tau)\iso \Z(\enmo_{\crp{F}T}V)$.
  Since we have
  $\enmo_{\crp{F}T}V \iso \enmo_{\crp{F}G}
  (V\tensor_{\crp{F}T}\crp{F}G)$, it follows
  $[\tau^G]_{\crp{F}}=[\tau]_{\crp{F}}$ and
  $\crp{F}(\tau^G)\iso \crp{F}(\tau)$.
  Since clearly $\crp{F}(\tau^G)\subseteq \crp{F}(\tau)$,
  we have $\crp{F}(\tau^G)=\crp{F}(\tau)$ as claimed.
\end{proof}
In general, $G_{\theta}$ may be smaller than $G_e=T$.
For $\xi\in \Irr(G_{\theta}\mid \theta)$, the field
$\crp{F}(\xi^{T})$ is  contained  in $\crp{F}(\xi)$, but may be
strictly smaller. If this happens, the Schur index of $\xi^T$ may be bigger
than that of $\xi$.
\begin{defi}\citep[Definition 1.1]{i81b}
  Let $K\nteq G$ and $\theta\in \Irr K$.
  We say that $\theta$ is \emph{semi-invariant} in $G$
  over the field $\crp{F}$,
  if for every $g\in G$ there is
  $\alpha\in \Gal(\crp{F}(\theta)/\crp{F})$ such that
  $\theta^g = \theta^{\alpha}$.
  If $\theta$ is semi-invariant over $\rats$, then we  say
  it is semi-invariant.
\end{defi}
\begin{lemma}
   The following assertions are equivalent:
  \begin{enumequiv}
  \item $\theta$ is semi-invariant over $\crp{F}$ in $G$.
  \item The  idempotent $e_{ (\theta, \crp{F}) }$  is invariant in $G$.
  \item A simple $\crp{F}K$\nbd module whose character contains
        $\theta$ as constituent is invariant in $G$.
  \end{enumequiv}
\end{lemma}
\begin{proof}
   The equivalence between (i) and (ii) is clear.

   Let $V$ be a simple $\crp{F}K$\nbd module.
   The character of $V$
   contains $\theta$ if and only if
   $Ve_{ (\theta, \crp{F}) }=V$ (by Lemma~\ref{l:Traceidempotent}),
   and $V$ is determined uniquely up to
   isomorphism by this property, since
   $\crp{F}Ke_{ (\theta, \crp{F}) }$ is artinian simple.
   The equivalence between (ii) and (iii) follows.
\end{proof}
\begin{lemma}\label{l:fnsg}
  Let $K\nteq G$ and $\theta\in \Irr K$ be
  semi-invariant over $\crp{F}$.
  Set $\Gamma=\Gal(\crp{F}(\theta)/\crp{F})$.
  \begin{enums}
  \item \label{i:fnsggal}For every $g\in G$ there is a unique
        $\alpha_g\in \Gamma$ such that
        $\theta^{g \alpha_g}=\theta$.
  \item \label{i:fnsghom}The map $g\mapsto \alpha_g$ is a group homomorphism from
        $G$ into $\Gamma$ with kernel $G_{\theta}$.
  \item \label{i:centchgalois}
        For $g\in G$ and $z\in \Z(\crp{F}K )$ we have
        $\omega_{\theta}(z^g) = \omega_{\theta}(z)^{\alpha_g}$,
        where
        \[\omega_{\theta}\colon \Z(\crp{F}K)\to \crp{F}(\theta)
        \]
        is the central character associated with $\theta$.
  \end{enums}
\end{lemma}
\begin{proof}
  Assertions~\ref{i:fnsggal} and~\ref{i:fnsghom} are proved in a paper of
  Isaacs~\citep[Lemma~2.1]{i81b}.

  Let $D\colon K\to \mat_{d}(\alcl{\crp{F}})$ be a
  representation affording $\theta$.
  Then $D^g$, defined by $D^g(k^g)=D(k)$ for $k\in K$,
  affords $\theta^g$.
  For $\alpha\in \Gal(\alcl{\crp{F}}/\crp{F})$
  we may define a representation $D^{\alpha}$ by letting
  $\alpha$ act on the matrix entries; it is clear that $D^{\alpha}$
  affords $\theta^{\alpha}$.
  Now $\omega_{\theta}$ is defined by
  $D(z)= \omega_{\theta}(z)I$ for
  $z\in \Z(\alcl{\crp{F}}K)$.
  We may extend $\alpha_g$ to $\alpha\in \Aut(\alcl{\crp{F}})$. Then
  \begin{align*}
    \omega_{\theta}(z^g)I
      &= \omega_{\theta^{\alpha g}}(z^g)I
       = D^{\alpha g}(z^g)
       = D^{\alpha}(z).
  \end{align*}
  If $z\in \Z(\crp{F}K)$, then
  $D^{\alpha}(z)
    = D(z)^{\alpha}
    = \omega_{\theta}(z)^{\alpha}I$.
  Thus~\ref{i:centchgalois} follows.
\end{proof}
\section{Magic crossed representations}\label{sec:magiccrossed}
In this section, we assume the following situation:
\begin{hyp}\label{h:bconf2}
Let $G$ be a group, $K\nteq G$ and $H\leq G$ with $G=HK$ and set
$L=H\cap K$.
Let $\alcl{\crp{F}}$ be an algebraically closed field and
let $\phi\in \Irr_{\alcl{\crp{F}}} L$ and $\theta\in \Irr_{\alcl{\crp{F}}} K$ be
characters of simple, projective modules over $\alcl{\crp{F}}L$ respective $\alcl{\crp{F}}K$
and $\crp{F}\subseteq \alcl{\crp{F}}$ a field such that the following conditions hold:
\begin{enums}
\item $n=(\theta_L, \phi)>0$.
\item $\crp{F}(\phi)= \crp{F}(\theta)$.
\item For every $h\in H$ there is
      $\gamma= \gamma_h\in  \Gal( \crp{F}(\phi)/\crp{F})$ such that
      $\theta^{h\gamma}= \theta$ and
      $\phi^{h\gamma} = \phi$.
\end{enums}
The field may have characteristic $p>0$, then the conditions imply
that $\phi$ and $\theta$ have $p$\nbd defect zero.
\end{hyp}
This situation typically arises when
$\theta$ and $\phi$ correspond under some ``natural''
correspondence.
We mention two examples:
\begin{examples}\label{es:conf2}\hfill
\begin{enums}
\item Suppose $\phi\in \Irr L$ and $\theta\in \Irr K$ are fully ramified with respect
      to each other. This means that $\theta$ vanishes on
      $K\setminus L$ and $\theta_L= n\phi$ with $n^2=\abs{K:L}$.
      Equivalently, $e_{\theta} = e_{\phi}$.
      It is clear that then $\rats(\theta)=\rats(\phi)$.
      Given $H$ as in Hypothesis~\ref{h:bconf2},
      we see that $\theta$ is semi-invariant in $G$ if and only if
      $\phi$ is semi-invariant in $H$.
      If this is the case, Hypothesis~\ref{h:bconf2} is true
      for $\crp{F}=\rats$.
\item Let $\pi$ be a set of primes and
      suppose that $G$ is $\pi$\nbd separable.
      Let $K= \grO_{\pi'}(G)$ and
      $N/K = \grO_{\pi}(G/K)$.
      Let $P$ be a
      Hall $\pi$\nbd subgroup of $N$
       and set $L=\C_K(P)$.
      Then for $H=\N_G(P)$ we have $G=HK$ and $L=H\cap K$.
      Glauberman-Isaacs correspondence defines a bijection between
      $(\Irr_{\alcl{\crp{F}}}(K))^P$ and $\Irr_{\alcl{\crp{F}}}L$.
      Here, $\crp{F}$ may be a field of characteristic zero or
      of characteristic $p$ with $p\in \pi$.
      By its naturalness, this correspondence commutes
      with field and group automorphisms.
      Thus if $\theta\in (\Irr_{\alcl{\crp{F}}}(K))^P$ and
      $\phi\in \Irr_{\alcl{\crp{F}}}(L)$ correspond
      and are semi-invariant in $H$,
      then Hypothesis~\ref{h:bconf2} holds.
            We will study a generalization of the case $\pi=\{p\}$ in
      Section~\ref{sec:aboveglaub}.
\end{enums}
\end{examples}
Now assume that Hypothesis~\ref{h:bconf2} holds.
As we go along, we will introduce further notation,
which is meant to carry through.
\begin{lemma}
\[i =\sum_{\gamma\in \Gal( \crp{F}(\phi)/\crp{F})}
e_{\phi}^{\gamma}e_{\theta}^{\gamma}
\]
is a $H$\nbd stable nonzero idempotent in $\crp{F}Ke_{ (\theta, \crp{F}) }$, and we have
$e_{ (\theta, \crp{F}) }i = i = ie_{ (\theta, \crp{F}) }$ and
$e_{(\phi,\crp{F})}i = i = ie_{ (\phi, \crp{F}) }$.
\end{lemma}
\begin{proof}
  We have $e_{ (\theta, \crp{F}) }= \sum_{\gamma\in \Gal( \crp{F}(\theta)/\crp{F})}
  e_{\theta}^{\gamma}$.
  Thus $e_{ (\theta, \crp{F}) }i = i=ie_{ (\theta, \crp{F}) }$ follows from
  $e_{\theta}^{\gamma}e_{\theta}^{\gamma'} = 0$
  for $\gamma\neq \gamma'\in \Gal(\crp{F}(\theta)/\crp{F})$.
  A similar argument shows that $e_{ (\phi, \crp{F}) }i=i= ie_{ (\phi, \crp{F}) }$.
  By assumption, $e_{\phi}e_{\theta}\neq 0$ and thus
  $i\neq 0$, as $e_{\theta}i = e_{\phi}e_{\theta}$.
  For $h\in H$, we have
  \[ i^h = \sum_{\gamma} e_{\phi}^{h\gamma}e_{\theta}^{h\gamma}
         = \sum_{\gamma} e_{\phi}^{\gamma_h\gamma} e_{\theta}^{\gamma_h\gamma}
         = i\]
  as desired.
\end{proof}
As $\crp{F}Ke_{ (\theta, \crp{F}) }$ is simple, it follows that
$\crp{F}Ki\crp{F}K = \crp{F}Ke_{ (\theta, \crp{F}) } $ and thus
$i\crp{F}Ki$ and $\crp{F}Ke_{ (\theta, \crp{F}) }$ are
Morita\index{Morita equivalence} equivalent.
\begin{lemma}\label{l:centerisos}
  \[ \Z(i\crp{F}Ki)  \iso  \Z(\crp{F}Ke_{ (\theta, \crp{F}) })
                 \iso \crp{F}(\theta) = \crp{F}(\phi)
                 \iso \Z(\crp{F}Le_{ (\phi, \crp{F}) })
  \]
  as fields with  $H$\nbd action.
\end{lemma}
\begin{proof}
  Since
  $\crp{F}K e_{ (\theta, \crp{F}) }
   = \crp{F}K i\crp{F}K$, the map
  $z\mapsto zi$ is an isomorphism between
  $\Z(\crp{F}K e_{ (\theta, \crp{F}) } ) $ and $\Z(i\crp{F}K i)$.
  It commutes with the action of $H$, as $i$ is invariant in $H$.

  The central character
  $\omega_{\theta}$ restricts to an isomorphism
  $\Z(\crp{F}K e_{ ( \theta, \crp{F}) })\iso \crp{F}(\theta)$
  and commutes with the action of $H$ by
   Part~\ref{i:centchgalois} of Lemma~\ref{l:fnsg}.
  The same reasoning applies to $\phi$.
   This completes the proof.
\end{proof}
(Alternatively, one can prove this lemma using
 Lemma~\ref{l:fe_iso}.)
\begin{lemma}\label{l:StensorFLfiFKi}
Set $Z= \Z(i\crp{F}Ki)$ and
let $S= (i\crp{F}Ki)^L$.
Then $S$ is a simple subalgebra of $i\crp{F}Ki$ with
 center $Z \iso \crp{F}(\theta)$, and dimension $n^2$ over $Z$,
and
\[\C_{i\crp{F}Ki}(S) = \crp{F}Li \iso \crp{F}Le_{ (\phi, \crp{F}) }.\]
\end{lemma}
\begin{proof}
  Set $i_0 = e_{\phi}e_{\theta}$.
  The following diagram is commutative:
  \[ \begin{CD}
       \crp{F}L e_{ (\phi, \crp{F}) } @>{\cdot i}>> i\crp{F}K i @>{\subseteq}>>
       \crp{F}Ke_{ (\theta, \crp{F}) }\\
       @V{\cdot e_{\phi}}VV   @V{\cdot i_0}VV
       @V{\cdot e_{\theta}}VV \\
       \crp{F}(\theta)Le_{\phi} @>{\cdot i_0}>>
       i_0\crp{F}(\theta)K i_0
       @>{\subseteq}>> \crp{F}(\theta)Ke_{\theta}.
     \end{CD}
  \]
  (Note that $ie_{\theta} = i e_{\phi}= ii_0 = i_0$.)
  By Lemma~\ref{l:fe_iso}, its vertical maps are isomorphisms.
  The result now follows from
  Lemma~\ref{l:ctcentr}.
\end{proof}
As in Section~\ref{sec:magic}, we have that
$i\crp{F}Ki \iso S\tensor_{Z} \crp{F}Le_{(\phi,\crp{F})}$.
But now $H$ may act nontrivially on $Z$, so $Z$ is in general not in the
center of $i\crp{F}Gi$.
To be precise, we have the following:
\begin{remark}
$\Z(i\crp{F}Gi)\cap Z = Z^{H}$.
\end{remark}
\begin{proof}
  Clear since
  $i\crp{F}Gi = \sum_{h\in H} i\crp{F}Ki h$.
\end{proof}
What we need is a subalgebra $S_0$ of $S$ such
that $\Z(S_0)=Z^H =: Z_0$ and $S=S_0 Z$.
We now add to Hypothesis~\ref{h:bconf2} the assumption that
there is such a subalgebra $S_0$ in $S$.
For example, if $S\iso \mat_n(Z)$, then
$\mat_n(Z_0)$ is such a subalgebra in $\mat_n(Z)$, and its image
in $S$ under some isomorphism is such a subalgebra in $S$.
In this example, $S_0$ depends on the choice of a particular
isomorphism $S\iso\mat_n(Z)$.

Even worse, there may be different non-isomorphic algebras $S_0$
of this type. For example, in $S=\mat_2(\compl)$ with
$\compl^H=\reals$ on can choose
\[ S_0 = \mat_2(\reals)
   \quad \text{or}\quad
   S_0 = \left\{
           \begin{pmatrix}
             a & b \\
            -\bar{b} & \bar{a}
           \end{pmatrix}
           \bigm| a, b \in \compl
         \right\}.\]
The following general result is clear in view of
Lemma~\ref{l:central}.
\begin{lemma}\label{l:s0galois}
  Let $Z_0\leq Z$ be a Galois extension of fields, and let
   $S$ be a  central simple algebra with $\Z(S)=Z$.
   Suppose $S_0\leq S$ with $\Z(S_0)=Z_0$ and
   $S=S_0 Z$.
   Then $S\iso S_0\tensor_{Z_0}Z$ and
   $\Gal(Z/Z_0)$ acts on $S$ via
   $(sz)^{\gamma}= s z^{\gamma}$.
\end{lemma}
There is a converse: If
$\Gal(Z/Z_0)$ acts on $S$ such that the action extends the natural
action of $\Gal(Z/Z_0)$ on $Z=\Z(S)$, then
$S_0=\C_S(\Gal(Z/Z_0))$ is central simple with center $Z_0$ and
$S=S_0 Z$~\citep[Lemma~1.2]{hochs50}.
We will not need this, however.

We return to the situation of Hypothesis~\ref{h:bconf2}.
Note that, if $S_0$ is given, we have two actions of $H$ on $S$:
The action by conjugation, and an action $\eps$ defined
by $(s_0 z)^{\eps(h)} = s_0 z^h$.
The second action has kernel $H_{\phi}$, since
the map $H\to \Gal(Z/Z_0)$ has kernel $H_{\phi}$.

We have now collected all the ideas necessary to generalize
the results of Sections~\ref{sec:magic} and~\ref{sec:corr}
to the situation of Hypothesis~\ref{h:bconf2}.
In particular, the reader will see that $i$ and $S$ are the
correct
objects to work with (and not the idempotent
$e_{ (\theta, \crp{F}) }e_{ (\phi, \crp{F}) }$, for example).
\begin{lemma}\label{l:crcoc}
  Suppose Hypothesis~\ref{h:bconf2} and let $S_0\subseteq S$ with
  $\Z(S_0)=Z_0$ and $S=S_0Z$.
  Define $\eps\colon H\to \Aut S$ by
  $(s_0 z)^{\eps(x)} = s_0 z^x$.
  For any $x\in H/L$ there is $\sigma(x)\in S$ such that
  for every $s\in S_0$ we have
  $s^x = s^{\sigma(x)}$.
  For $x,y \in H/L$ we have
  \[ \sigma(x)^{\eps(y)}\sigma(y) = \alpha(x,y)\sigma(xy)
     \text{ for some } \alpha(x,y)\in Z^{*},\]
  and $\alpha \in Z^2(H/L, Z^{*})$.
\end{lemma}
We emphasize that $S_0$ need not be invariant under $H$.
\begin{proof}
  Let $x\in H$.
  The map $s_0 z\mapsto s_0^x z$ (for $s_0\in S_0$ and $z\in Z$)
  yields a well defined $Z$\nbd algebra automorphism of $S$.
  By the
  Skolem-Noether theorem it is inner.
  This means that there is $\sigma(x)\in S^*$ such that
  ${s_0}^{x} = {s_0}^{\sigma(x)}$ for all $s_0\in S_0$.
  Since $L$ acts trivially on $S$, we may choose a map
  $\sigma\colon H/L \to S^{*}$ such that
  $s_0^h = s_0^{\sigma(Lh)}$ for all $s_0\in S_0$.

  Note that then for $s_0\in S_0$ and $z\in Z$, we have
  \[ (s_0z)^x = s_0^x z^x = s_0^{\sigma(x)}z^x
    =( s_0 z^x)^{\sigma(x)} = (s_0z)^{\eps(x)\sigma(x)}.\]
  Thus
  \[ s_0^{\sigma(xy)} = (s_0^x)^y
                      = (s_0^{\sigma(x)})^y
                      = (s_0^{\sigma(x)})^{\eps(y)\sigma(y)}
                      = s_0^{\sigma(x)^{\eps(y)}\sigma(y)}.\]
  Since $\C_S(S_0)=Z$, it follows that
  $\sigma(x)^{\eps(y)}\sigma(y) =\alpha(x,y)\sigma(xy)$
  for some $\alpha(x,y)\in Z^{*}$.
  From
  \[ \Big( \sigma(x)^{\eps(y)} \sigma(y)
     \Big)^{\eps(z)} \sigma(z)
     =
     \sigma(x)^{\eps(yz)}
     \Big( \sigma(y)^{\eps(z)} \sigma(z)
     \Big)
  \]
  it follows that
  \[\alpha(x,y)^z \alpha(xy,z)
    =
    \alpha(x,yz)\alpha(y,z).
  \]
  Thus $\alpha\in Z^2(H/L, Z^*)$.
\end{proof}
We may call $\sigma\colon H/L \to S$ a ``crossed projective
representation''. If we want to be more precise,
we speak of an
$\eps$\nbd crossed projective representation
or a projective representation which is crossed with respect to
$S_0$.
Note that $S_0 = \C_S(\eps(H))$, so that
$S_0$ is determined by $\eps$.
Conversely, it is clear that $\eps$ is determined by $S_0$.

If $S_0$ is fixed, then $\sigma(x)$ is unique up to multiplication
with elements of $Z$, and thus the image of
$\alpha$ in $H^2(H/L,Z^{*})$ is independent of the particular
choice of $\sigma$.
  The image of $\alpha$ in $H^2(H/L, Z^{*})$ \emph{does} depend on
  the choice of $S_0$, however.
On the positive side, we have:
\begin{lemma}\label{l:isoclass_s0}
  The class of $\alpha$ in $H^2(H/L,Z^{*})$ depends only on
  the isomorphism class of $S_0$.
\end{lemma}
\begin{proof}
  Suppose $S_0$ and $T_0$ are isomorphic. Since
  $S\iso S_0\tensor_{Z_0} Z \iso T_0\tensor_{Z_0} Z$, there is
  (by the Skolem-Noether theorem)
  $u\in S^{*}$ such that $T_0 = (S_0)^u$.

  For $h\in H$, define $\delta(h)\in \Aut S$ by
  $(t_0 z)^{\delta(h)}= t_0 z^h$ for $t_0\in T_0$ and $z\in Z$.
  Observe that
  $s^{\delta(h)}= s^{u^{-1}\eps(h)u}$.

  Now for $x\in H/L$, set
  \[ \tau(x) = [u,\eps(x)]\sigma(x)= u^{-1}u^{\eps(x)}\sigma(x).\]
  Then for $t_0=s_0^u\in T_0$ we have
  \[t_0^{\tau(x)}= s_0^{u u^{-1}u^{\eps(x)}\sigma(x)}
      = (s_0^u)^{\eps(x)\sigma(x)}
      = (s_0^u)^{x} = t_0^x
  \]
  and
  \begin{align*}
    \tau(x)^{\delta(y)}\tau(y)
     &= \left( u^{-1}u^{\eps(x)}\sigma(x)
        \right)^{u^{-1}\eps(y)u}
        \:
        u^{-1}u^{\eps(y)}\sigma(y)
     \\
     &= u^{-1} u^{\eps(xy)}\sigma(x)^{\eps(y)}\sigma(y)
      =  \alpha(x,y) \tau(xy).
  \end{align*}
  Thus $\sigma$ and $\tau$ define the same cocycle.
\end{proof}
\begin{defi}
  In the situation of Lemma~\ref{l:crcoc},
  we call $\sigma\colon H/L \to S$ a
  \emph{magic} ($\eps$-) crossed
  representation for the configuration of Hypothesis~\ref{h:bconf2}
  (with respect to $S_0$),
  if
  \begin{enums}
  \item $  \sigma(x)^{\eps(y)}\sigma(y) = \sigma(xy)$ for all
       $x,y \in H/L$ and
  \item $s^x = s^{\eps(x)\sigma(x)}$ for all $x\in H/L $ and
        $s\in S$.
  \end{enums}
\end{defi}
It is possible that for one choice of $S_0$, there exists a magic
crossed representation, while for another choice such a magic
crossed representation does not exist.
For more details on this question, see~\citep[2.37-2.41]{ladisch09diss}.
\begin{thm}\label{t:c-isosemi}
  Assume Hypothesis~\ref{h:bconf2}, and let
  $S_0\subseteq S$ with $S=\Z(S)S_0$ and
  $\Z(S_0)= \Z(S)^H$. Define $\eps\colon H/H_{\phi}\to \Aut S$
  by $(s_0 z)^{\eps(h)}= s_0 z^h$ for $s\in S_0$ and $z\in \Z(S)$.
  Let $\sigma\colon H/L \to S$
  be a  magic $\eps$\nbd crossed representation.
  Then the linear map
  \[ \kappa\colon \crp{F}H  \to C=\C_{i\crp{F}Gi}(S_0),\quad
     \text{defined by} \quad
     h \mapsto  h\sigma(Lh)^{-1} \text{ for } h\in H ,\]
  is an algebra-homomorphism and induces an isomorphism
  $ \crp{F}H e_{(\phi,\crp{F})} \iso C$.
\end{thm}
\begin{proof}
   Let $c_h = h \sigma(Lh)^{-1}$.
   It is easy to see that indeed $c_h\in C$.
  We compute
  \begin{align*}
    c_hc_g &= h\sigma(Lh)^{-1}g\sigma(Lg)^{-1}
            = hg \left( \sigma(Lh)^{-1}
                 \right)^{\eps(g)\sigma(Lg)
                         }
                 \sigma(Lg)^{-1}  \\
           &= hg \sigma(Lg)^{-1} \left( \sigma(Lh)^{\eps(g)}
                                 \right)^{-1}
            = hg \left( \sigma(Lh)^{\eps(g)} \sigma(Lg)
                 \right)^{-1}  \\
           &= hg \sigma(Lhg)^{-1} = c_{hg}.
  \end{align*}
  Thus $\kappa$ is an algebra homomorphism.
  From $\sigma(1)^{\eps(1)}\sigma(1)=\sigma(1)$ we see that
  $\sigma(1)= 1_S = i$, and thus
  $\kappa(l)= li$ for all $l\in L$.
  Thus
  $\kappa(\crp{F}Le_{ (\phi, \crp{F}) }) = \crp{F}Li$ and $\kappa(f)=0$ for
  every
  central primitive idempotent $f$ of $\crp{F}L$
  different from $e_{ (\phi, \crp{F}) }$.
  Now it follows as in the proof of Theorem~\ref{t:C-iso}
  that $\kappa$ induces an isomorphism from $\crp{F}He_{ (\phi, \crp{F}) }$
  onto $C$.
\end{proof}
\begin{cor}\label{c:corrsemi}
  Assume Hypothesis~\ref{h:bconf2} and let
  $\sigma\colon  H/L \to S $ be a magic
  crossed representation,
  with respect to $S_0 \subseteq S$.
  Then $i\crp{F}Gi \iso S_0 \tensor_{Z_0} \crp{F}He_{ (\phi, \crp{F}) }$.
  If $S_0 \iso \mat_n(Z_0)$, then
  $i\crp{F}Gi \iso \mat_n(\crp{F}H e_{ (\phi, \crp{F} ) })$ and
  $\crp{F}Ge_{ (\theta, \crp{F}) }$ and $\crp{F}He_{ (\phi, \crp{F}) }$ are
  Morita equivalent.
\end{cor}
\begin{proof}
  All assertions follow from the first.
  Let $C= \C_{i\crp{F}Gi}(S_0)$. Then by Lemma~\ref{l:central}
  we have $i\crp{F}Gi \iso S_0 \tensor_{Z_0} C$
  (Remember that $Z_0\subseteq\Z(i\crp{F}Gi)$).
  By Theorem~\ref{t:c-isosemi},
  $C\iso\crp{F}He_{(\phi, \crp{F})}$.
  The result follows.
\end{proof}
Suppose $\crp{F}\leq \compl$. Let us write
$\Irr(G\mid e_{(\theta,\crp{F})})$ for the set of all irreducible
characters $\chi\in \Irr G$ such that
$\chi(e_{(\theta,\crp{F})})\neq 0$.
(This notation could be used for arbitrary idempotents $e\in \compl G$.)
Of course,
\[\Irr(G\mid e_{(\theta, \crp{F})})
  = \bigcup_{ \alpha \in \Gal( \crp{F}(\theta) / \crp{F} )
             }
             \Irr(G\mid \theta^{\alpha}).
\]
\begin{thm}\label{t:semicorrchars}
  Assume Hypothesis~\ref{h:bconf2} with $\crp{F}\leq
  \compl$.
  Every magic crossed representation $\sigma$ defines
  linear
  isometries
  \[
    \iota=\iota(\sigma)\colon
   \compl[\Irr(U\mid  e_{(\theta,\crp{F})})]
   \to
   \compl[\Irr(U\cap H\mid e_{(\phi, \crp{F})})]
   \quad
    (K\leq U\leq G)
  \]
  with Properties~\ref{i:c_irr}--\ref{i:c_brauer} from
  Theorem~\ref{t:corr}
  (where $S$ has to be replaced by $S_0$
   in Property~\ref{i:c_brauer}).
\end{thm}
\begin{proof}
  Let us first assume that $\crp{F}(\phi)^H= \crp{F}$.
  Then $Z_0 = \crp{F}1_S= \crp{F}i$ and
  $i\crp{F}G i \iso S_0 \tensor_{\crp{F}} C$,
  where $C =\C_{i\crp{F}Gi}(S_0)$.
  The magic crossed representation $\sigma$ defines an isomorphism
  $\crp{F}H e_{ (\phi, \crp{F}) } \to C$.
  By scalar extension, we get an isomorphism
  $\compl H e_{ (\phi, \crp{F}) }\to \compl \tensor_{\crp{F}}C$.
  Note that
  $\compl \tensor_{\crp{F}}C =
    \C_{ \compl \tensor_{\crp{F}} (i\crp{F}G i) }
      (\compl \tensor_{\crp{F}}S_0)$.
  Since $\compl \tensor_{\crp{F}}S_0$ is central simple, we get
  isomorphisms
   \begin{multline*}
     \ZF(\compl G e_{ (\theta,\crp{F}) },\compl)
       \xrightarrow[\text{(\ref{l:fullidem})}]{\text{Res}}
       \ZF(i\compl G i, \compl)
       \\
       \xrightarrow[\text{(\ref{l:cforms})}]{\eps}
       \ZF(\compl \tensor_{\crp{F}}C, \compl)
      \xrightarrow[\text{(\ref{t:c-isosemi})}]{\kappa^{*}}
     \ZF(\compl H e_{ (\phi, \crp{F}) }, \compl).
  \end{multline*}
  As before, for $\chi\in \compl[\Irr(G\mid e_{ (\theta, \crp{F}) })]$
  we have
  $\chi^{\iota}(h) = \chi( s_0 \sigma(Lh)^{-1}h)$, where
  $s_0\in S_0$ is some element with
  $\tr_{S_0/\crp{F}}(s_0)=1$.
  The rest of the proof of Theorem~\ref{t:corr} now carries over verbatim.

  Now drop the assumption that
  $\crp{F}(\phi)^H= \crp{F}$ and set
  $\crp{E}_0= \crp{F}(\phi)^H$.
  (Thus $\crp{E}_0\iso Z_0=\Z(S)^H$.)
  Then
  \[\crp{E}_0=\crp{F}(\sum_{ g\in [G:G_{\theta}]}\theta^g)
          =\crp{F}(\sum_{ h\in [H:H_{\phi}]  } \phi^h),
  \]
  and $\crp{F}(\chi)$ contains $\crp{E}_0$ for any
  $\chi\in \Irr(G\mid e_{ (\theta, \crp{F}) })
    \cup\Irr(H\mid e_{ (\phi, \crp{F}) })$.
  Also $\phi$ and $\theta$ remain semi-invariant over $\crp{E}_0$.

  Observe that
  \[ e_{(\theta,\crp{F})}= \sum_{ \alpha\in \Gal(\crp{E}_0/\crp{F}) }
                               e_{(\theta, \crp{E}_0)}^{\alpha}
    \quad \text{and}\quad
     e_{(\phi, \crp{F})} = \sum_{ \alpha\in \Gal(\crp{E}_0/\crp{F}) }
                               e_{ (\phi, \crp{E}_0)}^{\alpha}.\]
   All the idempotents $e_{(\theta, \crp{E}_0)}^{\alpha}$
   and
   $e_{ (\phi, \crp{E}_0)}^{\alpha}$ are invariant in $H$.
   Set
   \[ j = \T_{\crp{E}_0}^{\crp{F}(\theta)}(e_{\phi}e_{\theta})
        = \sum_{\alpha\in \Gal(\crp{F}(\theta)/\crp{E}_0)}
        (e_{\phi}e_{\theta})^{\alpha}.
   \]
   Then $j$ is invariant in $H$ and
   $i= \T_{\crp{F}}^{\crp{E}_0}(j)$.
   The following diagram is commutative:
   \[ \begin{CD}
       \crp{F}He_{(\phi, \crp{F})}
       @>{{}\cdot i}>>
       i\crp{F}G i
       @>\subseteq >>
       \crp{F}Ge_{(\theta,\crp{F})}
       \\
       @V{{}\cdot e_{(\phi,\crp{E}_0)} }VV
       @V{}\cdot j VV
       @V{{}\cdot e_{(\theta,\crp{E}_0)}}VV
       \\
       \crp{E}_0He_{(\phi, \crp{E}_0)}
       @>{}\cdot j >>
       j\crp{E}_0G j
       @>\subseteq >>
       \crp{E}_0Ge_{(\theta, \crp{E}_0)}
   \end{CD}\]
   The vertical maps are isomorphisms
   (by a generalization of~Lemma~\ref{l:fe_iso}).
   It follows that the map
   $H/L\ni x\mapsto \tau(x)=\sigma(x)j$ is a crossed magic representation
   $\tau\colon H/L \to Sj = (j\crp{E}_0K j)^L$
   with respect to $S_0j$.

  Let $\alpha\in \Gal(\crp{F}(\theta)/\crp{F})$.
  Then $\tau^{\alpha}$ defined by
  $x\mapsto \tau(x)^{\alpha} = \sigma(x)j^{\alpha}$ is a magic
  crossed representation for the configuration of
  $\theta^{\alpha}$ and $\phi^{\alpha}$.
  This follows from the fact that $\tau$ is magic, or
  from the above argument with $j^{\alpha}$,
  $e_{(\phi^{\alpha},\crp{E}_0)}$ and $e_{(\theta^{\alpha},\crp{E}_0)}$  instead
  of $j$, $e_{(\phi, \crp{E}_0)}$ and $e_{(\theta, \crp{E}_0)}$.

    By the first part of the proof, the maps
    $\tau^{\alpha}$ ($\alpha\in \Gal(\crp{F}(\theta)/\crp{F})$)
    determine isometries $\iota(\tau^{\alpha})$
  between $\compl[\Irr(G\mid e_{ (\theta^{\alpha},\crp{E}_0) })]$
  and $\compl[\Irr(H\mid e_{ (\phi^{\alpha}, \crp{E}_0) })]$ commuting
  with field automorphisms over $\crp{E}_0$.

  Choose
  $\alpha_i\in \Gal(\crp{F}(\theta)/\crp{F})$
  with
  \[ \Gal(\crp{F}(\theta)/\crp{F})
     = \bigdcup_{i=1}^{\abs{\crp{E}_0:\crp{F}}}
       \Gal(\crp{F}(\theta)/\crp{E}_0) \alpha_i.\]
  Thus
  \[ \Gal(\crp{E}_0/\crp{F})
     = \left\{
         {\alpha_i}_{|\crp{E}_0}
         \mid
         i = 1, \dotsc, \abs{\crp{E}_0:\crp{F}}
       \right\}\]
  and
  \[  \compl[\Irr(G\mid e_{(\theta, \crp{F})})]
       = \bigoplus_{i}
         \compl[\Irr(G\mid e_{(\theta^{\alpha_i},\crp{E}_0)})].\]
  We define
  \[\iota(\sigma)
     = \bigoplus_{i}\iota(\tau^{\alpha_i})\colon
       \compl[\Irr(G\mid e_{(\theta, \crp{F})})]
       \to \compl[\Irr(H\mid e_{(\phi, \crp{F})})].
     \]
   If $\alpha_{|\crp{E}_0}=\beta_{|\crp{E}_0}$, then
  $\tau^{\alpha}=\tau^{\beta}$ and
  $\iota(\tau^{\alpha})= \iota(\tau^{\beta})$.
  Thus $\iota(\sigma)$ is independent of the choice of the
  $\alpha_i$.
  It is clear that $\iota(\sigma)$ commutes with field
  automorphisms fixing $\crp{F}$.
  The isometries $\iota(\tau^{\alpha})$
  preserve Brauer
  equivalence classes of irreducibles characters over $\crp{E}_0$.
  Since $\crp{E}_0\subseteq \crp{F}(\chi)$ for all
  $\chi \in \Irr(G\mid e_{ (\theta, \crp{F}) })$, it follows that
  $\iota(\tau^{\alpha})$ preserves Brauer equivalence classes over
  $\crp{F}$.
  Thus $\iota(\sigma)$ has Property~\ref{i:c_brauer}.
  The other properties are clear. The proof is finished.
\end{proof}
\begin{remark}
  We explain the relation between the last result
  (respective the more special Theorem~\ref{t:corr})
  and Turull's theory of
  the Brauer-Clifford group~\citep{turull09}.
  The Brauer-Clifford group of a group $X$ and a
  commutative $X$\nbd algebra consists of equivalence classes of
  certain $X$\nbd algebras.
  In the situation of Hypothesis~\ref{h:bconf2}, there are defined
  elements $[[\theta]]$ and $[[\phi]]$ of the
  Brauer-Clifford group $\BrClif(H/L, \crp{F}(\theta))$
  of $H/L$ over $\crp{F}(\theta)$.
  The equivalence class $[[S]]$ of the
  $H/L$\nbd algebra $S$ also belongs to
  $\BrClif(H/L, \crp{F}(\theta))$, and one can show that
  $[[\theta]]= [[S]][[\phi]]$ in the Brauer-Clifford
  group~\citep[Theorem~A.35]{ladisch09diss}.
  In particular, if $[[S]]= 1$, then a result of
  Turull~\citep[Theorem~7.12]{turull09}
  yields a character bijection as in
  Theorem~\ref{t:semicorrchars}.
  One can also show that
  $[[S]]=1$ if and only if
  $S\iso\mat_n(Z)$ and a magic crossed
  representation $H/L\to S$ exists,
  which is crossed with respect to a subalgebra
  $S_0\iso \mat_n(Z_0)$.
\end{remark}

Our next goal is to exhibit the relation between the
results of this section and those of
Sections~\ref{sec:magic} and~\ref{sec:corr}.
In Section~\ref{sec:magic}, our standing assumption
(Hypothesis~\ref{h:bconf}) was that
$\phi$ and $\theta$ are invariant in $H$ and that
the field $\crp{F}$ contains the values of $\phi$ and
$\theta$.
Under the assumptions which are in force in this section
(Hypothesis~\ref{h:bconf2}) this is not the case.
However, Hypothesis~\ref{h:bconf} holds for the configuration we
get if we replace the group $G$ by $G_{\theta}$, the subgroup $H$ by
$H_{\phi}$ and the field $\crp{F}$ by $\crp{F}(\phi)$.
Thus it makes sense to ask if there is a magic representation
$H_{\phi}/L \to ( e_{\phi}e_{\theta}\crp{F}(\phi)K
e_{\phi}e_{\theta})^L$
in the sense of
  Definition~\ref{d:magic}
for the configuration $(G_{\theta}, H_{\phi}, K,L,\theta, \phi)$
over the field
   $\crp{F}(\phi)$.
\begin{prop}\label{p:semicorr_corr}
  Assume Hypothesis~\ref{h:bconf2} and let
  $\sigma\colon H/L\to S$ be a magic crossed representation.
  Set $i_{\phi}=e_{\theta}e_{\phi}$.
  Then
  \[  \sigma_{\phi}\colon H_{\phi}/L
                          \to
                          T =(i_{\phi}\crp{F}(\phi)Ki_{\phi})^L,
     \quad
     h \mapsto
     \sigma_{\phi}(h)= \sigma(h)i_{\phi}
  \]
  is a magic representation for
  the configuration $(G_{\theta},H_{\phi},K,L,\theta,\phi)$ and for
  $\chi\in \Irr(G_{\theta}\mid \theta)$ we have
  \[ (\chi^G)^{\iota(\sigma)} = (\chi^{\iota(\sigma_{\phi})})^H.\]
\end{prop}
\begin{proof}
  For $s\in S$, we have $si_{\phi} =se_{\theta}$.
  That $\sigma_{\phi}$ is a magic representation
  follows from Lemma~\ref{l:fe_iso}.

  Since $\iota(\sigma)$ commutes with induction of characters,
  it suffices to show that
  $\chi^{\iota(\sigma)}= \chi^{\iota(\sigma_{\phi})}$.
  Choose $s_0\in S_0$ with
  $\tr_{S_0/Z_0}(s_0)=1 =\tr_{S/Z}(s_0)$.
  Then $\tr_{T/\crp{F}(\theta)}(s_0i_{\phi})=1$.
  Observe that $e_{\theta}i= i_{\phi}$.
  This yields that for arbitrary $h\in H_{\phi}$
  \begin{align*}
    \chi^{\iota(\sigma)}(h)
      &= \chi(s_0 \sigma(Lh)^{-1}h)
       = \chi(e_{\theta}s_0 \sigma(Lh)^{-1}h)
       \\
      &= \chi(s_0i_{\phi} \sigma_{\phi}(Lh)^{-1}h)
       = \chi^{\iota(\sigma_{\phi})}(h)
  \end{align*}
  as claimed.
\end{proof}
This result just means that we get the correspondence
$\iota(\sigma)$ by composing the
Clifford correspondences
associated with $\theta$ and $\phi$ and a correspondence induced by
a magic representation.
Note that $\iota(\sigma)$ is determined by $\iota(\sigma_{\phi})$ and this
property.

Conversely, if the configuration
$(G_{\theta}, H_{\phi}, K, L, \theta, \phi)$
admits a magic representation $\tau$, we may
compose the correspondence $\iota(\tau)$ with the Clifford
correspondences between
$\Irr(G\mid \theta)$ and $\Irr(G_{\theta}\mid \theta)$,
respective between
$\Irr(H\mid \phi)$ and $\Irr(H_{\phi} \mid \phi)$.
We get then a correspondence between
$\Irr(G\mid \theta)$ and $\Irr(H\mid \phi)$.
But we do not get compatibility with field automorphisms and Schur
indices over the field $\crp{F}$ from this argument, and we may
have $\crp{F}(\theta)\nsubseteq \crp{F}(\chi)$ for some
$\chi\in \Irr(G\mid \theta)$.

Finally, let us show that $\iota(\sigma)$ is
independent of the particular choice of $S_0$:
\begin{remark}\label{r:corrindep}
  Assume the situation of Theorem~\ref{t:semicorrchars} and
  let $u\in S^{*}$. Set $\tau(x)= u^{-1}u^{\eps(x)}\sigma(x)$.
  Then $\tau$ is a magic crossed representation with respect to
  $T_0=(S_0)^u$, and
  $\iota(\sigma)= \iota(\tau)$.
\end{remark}
\begin{proof}
  That $\tau$ is magic follows from the proof of
  Lemma~\ref{l:isoclass_s0}.
  Observe that
  $\tau(x)= u^{-1}\sigma(x)u^x$.
  Thus
  \begin{align*}
   \chi^{\iota(\tau)}(h)
    &= \chi( h \tau(Lh)^{-1} (s_0)^u)
     = \chi( h (u^h)^{-1}\sigma(Lh)^{-1}u (s_0)^u)
     \\
    &= \chi( u^{-1} h \sigma(Lh)^{-1} s_0 u)
     = \chi( h \sigma(Lh)^{-1}s_0)
     = \chi^{\iota(\sigma)}(h).
  \end{align*}
\end{proof}
Thus if we view the isomorphism type of $S_0$ as fixed, we may
choose some $S_0$ without loss of generality.
If the subalgebra $S_0$ is given, a magic representation
$\sigma$
is unique up to multiplication with a map
$\lambda\colon H/L \to Z^*$ such that
$\lambda(x)^y\lambda(y)=\lambda(xy)$ for all
$x$, $y\in H/L$.
(In other words, $\lambda\in Z^1(H/L, Z^*)$.)
In particular, $\lambda_{H_{\phi}} \in \Lin (H_{\phi}/L)$.
It is not difficult to see that
$\chi^{\iota(\lambda\sigma)}=
 \lambda^{-1} \chi^{\iota(\sigma)}$
for $\chi\in \Irr(G \mid \theta)$.
As such a $\chi$ vanishes on
$G\setminus G_{\theta}$, we see that
$\iota(\lambda \sigma)=\iota(\sigma)$, if
$\lambda_{H_{\phi}}=1$. In particular,
$\iota(\lambda \sigma)=\iota(\sigma)$ if
$\lambda \in B^1(H/L, Z^*)$, that is, if there is
$a\in Z^*$ such that
$\lambda(x) = a^{-1}a^x$ for all $x\in H/L$.

For the last remark in this section, note that
if $\sigma$ is a magic crossed representation, then
$x\mapsto \nr_{S/Z}(\sigma(x))$ defines an  element of
$Z^1(H/L, Z^*)$, which we denote simply by $\nr (\sigma)$.
(Remember that $\nr=\nr_{S/Z}$ denotes the reduced norm
 of $S$ with respect to $Z$.)
The following  is analogous to Remark~\ref{r:corrdet}:
\begin{remark}\label{r:semicorrdet}
  Let $\pi$ be the set of  prime divisors of $n$.
  If there is any magic crossed  representation,
  then there is a  magic crossed representation $\sigma$ such that
  the class of
  $\nr (\sigma)$ in $H^1(H/L, Z^*)$  has order a $\pi$\nbd number.
\end{remark}
\begin{proof}
  Suppose $\sigma \colon H/L \to S$ is given.
  Let $\lambda= \nr (\sigma)\in Z^1(H/L, Z^*)$.
  Let $b$ be the $\pi'$\nbd part of $\ord(\lambda B^1(H/L, Z^*))$.
  As $n = \dim \sigma$ is $\pi$, there is $r\in \ints$ with
  $rn+1 \equiv 0 \mod b$.
  Then $\nr( \lambda^r \sigma) = \lambda^{rn+1}$ has $\pi$\nbd order
  modulo $B^1(H/L, Z^*)$.
\end{proof}
Note that $Z^1(H/L, Z^*)$ might be infinite. If it is finite for
some reason, we can even get that $\nr(\sigma)$ itself has order a
$\pi$\nbd number.

\section[Coprime multiplicity]{An example: coprime multiplicity}\label{sec:coprim}
We need the following result, which can be found
in papers of Dade~\citep[Theorem~4.4]{dade81b} and
Schmid~\citep[Theorem~2]{schmid85}:
\begin{prop}\label{p:dadeschmid}
  Let $Z/Z_0$ be a Galois extension with Galois group $\Gamma$ and
  $S$ a central simple $Z$\nbd algebra such that
  every $\gamma\in \Gamma$ can be extended to
  an automorphism of $S$ (as ring).
   If the Schur index of $S$ is prime to
  $\abs{\Gamma}$, then there is $S_0\leq S$ with
  $\Z(S_0)=Z_0$ and $S= S_0Z \iso S_0\tensor_{Z_0}Z$.
   If $\abs{\Gamma}$ is prime to $\dim_Z S$, then
   $S_0$ is unique up to conjugacy (with elements of $S^{*}$).
\end{prop}
This can be derived from Teichm\"{u}ller's work on
noncommutative Galois theory, as Schmid~\citep{schmid85} has
pointed out. (Teichm\"{u}ller considered simple algebras such that
 $\Aut_{Z_0}S \to \Gal(Z/Z_0)$ is surjective.
 See~\citet{eilmac48} for an exposition and related results.)

\begin{prop}\label{p:coprimfie}
  Assume Hypothesis~\ref{h:bconf2} with $(n, \abs{H/L})=1$.
  Then there is a magic crossed representation $\sigma$
  such that its reduced norm is in $B^1(H/L, Z^*)$.
  It yields a canonical correspondence $\iota(\sigma)$
  between
  $\compl[\Irr( G\mid e_{(\theta,\crp{F})})]$
  and
  $\compl[\Irr( H \mid e_{(\phi, \crp{F})})]$.
\end{prop}
\begin{proof}
  As $\Gamma=\Gal(Z/Z_0)$ is a factor group of $H/L$, it follows that
  $\abs{\Gamma}$ and $ \dim_{\crp{F}(\phi)}S$ are coprime.
  Thus by  Proposition~\ref{p:dadeschmid}, there is
  $S_0\subseteq S$ with $\Z(S_0)=Z_0$ and $S_0Z=S$,
  and $S_0$ is  unique up to inner automorphisms of $S$.
  For the moment, fix $S_0$.
  By Lemma~\ref{l:crcoc} there is an
  $\eps$-crossed projective representation with factor set
  $\alpha\in Z^2(H/L, Z^*)$, say. But as
  $n$ is coprime to $\abs{H/L}$, it follows that the cohomology
  class of
  $\alpha $ is trivial. Thus there exists a magic crossed
  representation with respect to $S_0$.

  Since $(n, \abs{H/L})=1$, there is a
  magic crossed representation
  such that its reduced norm is in $B^1(H/L, Z^*)$
  (by Remark~\ref{r:semicorrdet} and since
  the exponent of $H^1(H/L, Z^*)$ divides $\abs{H/L}$).
  In particular, for $x\in H_{\phi}/L$ we have
  $\nr(\sigma(x))=1$.
  Let $i_{\phi}= e_{\phi}e_{\theta}$ and define
  $\sigma_{\phi}(x)= \sigma(x)i_{\phi}$ for
  $x\in H_{\phi}/L$ as in Proposition~\ref{p:semicorr_corr}.
  Then $\nr{\sigma_{\phi}(x)}=1$ for $x\in H_{\phi}/L$.
  The magic representation $\sigma_{\phi}$  is uniquely
  determined by this condition,
  and the correspondence $\iota(\sigma)$ is
  determined by $\sigma_{\phi}$ (by
  Proposition~\ref{p:semicorr_corr}).
  Thus the correspondence $\iota(\sigma)$ is canonical in the
  sense that it is independent of the choice of the particular
  map $\sigma$.

  It remains to show that $\iota(\sigma)$ is independent of the
  choice of $S_0$.
  So assume that instead of
  $S_0$ we work with $S_0^u$.
  Then
  $\tau(x)= u^{-1}u^{\eps(x)}\sigma(x)$ is a magic crossed
  representation that yields the same correspondence as $\sigma$,
  by Remark~\ref{r:corrindep}.
  Since
  \[\nr (\tau(x)) = \nr(u)^{-1}\nr(u)^x \nr(\sigma(x)),
  \]
  it follows that $\nr\tau \in B^1(H/L, Z^*)$ also.
\end{proof}
\begin{cor}
  In the situation of Proposition~\ref{p:coprimfie}, there is a
  uniquely defined character $\psi$ of $H_{\phi}/L$, and
  the correspondence $\iota$ has the following property: For
  $\chi\in \Irr(G_{\theta}\mid \theta)$, we have
  \[ \left(\chi_{H_{\phi}}\right)_{\phi} = \psi \chi^{\iota}
     \quad
     \text{and}
     \quad
     \left((\chi^{\iota})^{G_{\theta}} \right)_{\theta}= \overline{\psi}\chi.\]
\end{cor}
(We do not claim that $\psi$ is defined by these equations.)
\begin{proof}
  Let $\psi$ be the magic character of the magic representation
  $x \mapsto \sigma(x)e_{\phi}$. This defines $\psi$
  unambiguously.
  The result follows from Theorem~\ref{t:corr},
  in particular~\ref{i:c_res} and~\ref{i:c_ind}.
\end{proof}
It may be worth pointing out that if we assume
Hypothesis~\ref{h:bconf} instead of the more general
Hypothesis~\ref{h:bconf2}, then Proposition~\ref{p:coprimfie} is
an immediate corollary of the results from
Sections~\ref{sec:magic} and~\ref{sec:corr}.
One only needs to observe that the cocycle associated with a magic
representation must be trivial, since $\dim S$ and $\abs{H/L}$ are
coprime.
In particular, if one doesn't care about rationality questions and
works simply over $\compl$, one can give a
rather quick and transparent proof that there is a correspondence
between $\Irr(G\mid \theta)$ and $\Irr(H\mid \phi)$, if
$(\theta_L, \phi)$ and $\abs{H/L}$ are coprime.

That the Clifford extensions associated with the characters
$\theta$ invariant in $G$ and $\phi$ invariant in $H$ are
isomorphic was already proved by Dade~\citep[0.4]{dade70b} in a
more general situation,
 but over an
algebraically closed field. Schmid~\citep{schmid88} has generalized
Dade's result to arbitrary fields, under the additional assumption
that the Schur indices of $\theta$ and $\phi$ are coprime to $\abs{H/L}$.
One the other hand, both Dade and Schmid work in the more general
context of group graded algebras.
 The
description using the magic\index{Magic representation} character
$\psi$ seems to be new, however.

 It would be nice to have a purely character
theoretic description of the correspondences $\iota$. If $\psi$
vanishes nowhere , then $\chi^{\iota}$ can be computed from the
equation $\chi_{H_{\phi}}= \psi \chi^{\iota}$. In this case, one
needs only to know $\psi$, but not $\sigma$ or a special element
$s_0\in S_0$ with $\tr(s_0)=1$, to compute $\chi^{\iota}$. This is
true for example if $G/K$ is a $p$\nbd group:
\begin{remark}
  If $x\in H_{\phi}/L$ has $p$\nbd power order,
  where $p$ is any prime, then
  $\psi(x)\neq 0$
  (in the situation of
  Proposition~\ref{p:coprimfie}).
\end{remark}
\begin{proof}
  If $\mathfrak{P}$ is a prime ideal of the ring of algebraic integers
  in $\compl$ that lies above $p$, then
  $\psi(x)-\psi(1)\in \mathfrak{P}$.
  Since $\psi(1)=n \notin p\ints$, it follows
  $\psi(x)\notin \mathfrak{P}$, and so $\psi(x)\neq 0$.
\end{proof}

The proposition and the corollary apply in particular if
$n=1$. Then $S\iso \crp{F}(\phi)$ and the canonical choice of
$\sigma$ is the trivial map $\sigma(x)=i$ for all $x$.
In particular, for
$\chi\in \Irr(G_{\theta}\mid \theta)$ we have
$(\chi_{H_{\phi}})_{\phi} = \chi^{\iota}$.
The last equation in fact defines then
the correspondence. It follows that $\chi^{\iota}$ is the unique element in
$\Irr(H_{\phi}\mid \phi)$ with $(\chi_{H_{\phi}}, \chi^{\iota})\neq 0$ and the
correspondence can also be defined by this condition
(cf.\ Proposition~\ref{p:centcorr}).
This fact is  known and  can be proved just using elementary character
theory~\citep[Lemma~4.1]{i84}.
 Theorem~\ref{t:semicorrchars} also implies that the correspondence
   preserves Schur
indices if $n=1$: Because then clearly
$S_0\iso \crp{F}(\theta)^H$ is split.
In fact, $\crp{F}He_{(\phi, \crp{F}) } \iso i\crp{F}Gi$ when $n=1$, where the
isomorphism is given by
multiplication with $i$.
Note, however, that Hypothesis~\ref{h:bconf2} still involves a
rather special hypothesis about the fields
$\crp{F}(\theta)$ and $\crp{F}(\phi)$.

\section{Induction}\label{sec:induc}
\begin{lemma}\label{l:bconf_comp}
  Let $G$ be a finite group,
  $H\leq U\leq G$
  and $K\nteq G$.
  Set $N=K\cap U$ and $L=K\cap H$.
  Let $\theta\in \Irr K$, $\eta \in \Irr N$ and $\phi\in \Irr L$.
  Assume that the configurations
  $(G,U,K, N, \theta, \eta)$ and $(U,H,N,L, \eta,\phi)$
  fulfill the conditions of Hypothesis~\ref{h:bconf2}
  for the same field $\crp{F}$.
  Then the conditions of Hypothesis~\ref{h:bconf2} also hold for the
  ``composed'' configuration
  $(G,H,K,L,\theta, \phi)$ over $\crp{F}$.
\end{lemma}
\begin{figure}[ht]
\setlength{\unitlength}{0.45ex}
\centering
\begin{picture}(75,60)(-10,-4)
\put(23.5,2.5){\line(1,1){26}}
\put(18.5,2.5){\line(-1,1){5.5}}
\put(8.0,13.0){\line(-1,1){5.5}}
\put(2.5,23.5){\line(1,1){26}}
\put(49.5,33.5){\line(-1,1){5.5}}
\put(39.0,44.0){\line(-1,1){5.5}}
\put(13.0,13.0){\line(1,1){26}}
\put(19.3,-1){$L$}
\put(13,-2){$\phi$}
\put(-2.2,18.4){$K$}
\put(-7,17){$\theta$}
\put(29.4,50){$G$}
\put(50,30){$H$}
\put(8.2,8.4){$N$}
\put(3,8){$\eta$}
\put(39.7,39.7){$U$}
\end{picture}
\caption{A composed basic configuration}
\label{fig:combconf}
\end{figure}
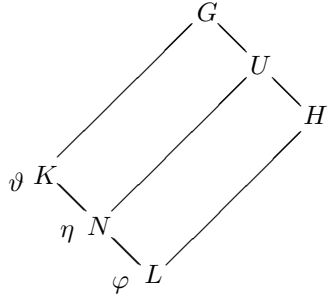
\begin{proof}
  Note that we assume in particular that $G=UK$  and $U= HN$.
  It follows that $G=HK$ and $L=H\cap K$.
  (See Figure~\ref{fig:combconf}.)
 By assumption,
 $(\theta_L, \phi)_L
   \geq (\theta_N, \eta)_N (\eta_L, \phi)_L
   >0$,
 and
  $\crp{F}(\theta)=\crp{F}(\eta)=\crp{F}(\phi)$.

  Let $h\in H$.
  By assumption, there is $\gamma_h\in
  \Gal(\crp{F}(\eta)/\crp{F})$ such that
  $\theta^{h\gamma_h}=\theta$ and $\eta^{h\gamma_h}=\eta$;
  and there is $\delta_h\in \Gal(\crp{F}(\phi)/\crp{F})$ such that
  $\eta^{h\delta_h}=\eta$ and $\phi^{h\delta_h}=\phi$.
  Since only the identity of
  $\Gal(\crp{F}(\eta)/\crp{F})$ can fix
  $\eta^h$, it follows that $\delta_h=\gamma_h$.
\end{proof}
In the situation of the lemma, suppose that we have
magic crossed representations
$\sigma_1$ and $\sigma_2$ for the configurations
$(G,U,K, N, \theta, \eta)$ and $(U,H,N,L, \eta,\phi)$,
respectively.
Then we have isometries
\begin{align*}
  \iota(\sigma_1)\colon
  &
  \compl[\Irr(G\mid e_{(\theta,\crp{F})})]
  \to
  \compl[\Irr( U \mid e_{(\eta, \crp{F})})]
  \quad \text{and} 
  \\
  \iota(\sigma_2)\colon
  &
  \compl[\Irr( U\mid e_{(\eta,\crp{F})})]
  \to
  \compl[\Irr( H\mid e_{(\phi, \crp{F})})]
\end{align*}
which commute with field automorphisms over $\crp{F}$ and have the
other properties of Theorem~\ref{t:semicorrchars}.
It follows that
\[\iota(\sigma_1)\iota(\sigma_2)
  \colon
  \compl[\Irr(G\mid e_{(\theta,\crp{F})})]
  \to
  \compl[\Irr( H\mid e_{(\phi, \crp{F})})]
\]
is an isometry as in Theorem~\ref{t:semicorrchars}.
The question arises if this isometry comes from a magic
crossed representation.
We try to prove this now.

Define
\begin{align*}
  i_1 &= \sum_{\gamma}(e_{\theta}e_{\eta})^{\gamma},
  &
  i_2 &= \sum_{\gamma}(e_{\eta}e_{\phi})^{\gamma},
  \\
  i   &= \sum_{\gamma}(e_{\theta}e_{\phi})^{\gamma},
  &
  j   &= \sum_{\gamma}(e_{\theta}e_{\eta}e_{\phi})^{\gamma},
\end{align*}
where all sums run over $\gamma\in \Gal(\crp{F}(\theta)/\crp{F})$.
Then
\[
  j = i_1 i_2 = i i_1 = i i_2 = i j = ji.
\]
Set
\[ S_1 = (i_1\crp{F}K i_1)^N, \quad
   S_2 = (i_2\crp{F}N i_2)^L
   \quad \text{and}\quad
   S = (i \crp{F}K i)^L.
   \]
Then $\Z(S)\iso \Z(S_1)\iso \Z(S_2)\iso \crp{F}(\phi)$.
All of the three centers are isomorphic to
$\Z(jSj)$ via $z\mapsto zj$.
\begin{lemma}\label{l:s1tensors2}
The map
$ s_1 \tensor s_2 \mapsto
   s_1s_2 \in S$
defines an isomorphism
\[S_1\tensor_{\Z(S)} S_2 \iso jSj.\]
\end{lemma}
\begin{proof}
Clearly, the map is well defined.
Since $S_1$ and $S_2$ commute, it is a ring homomorphism.
The map must be injective since both
$S_1$ and $S_2$ are central simple over $\Z(S)$.
To show that the image is all of $jSj$, it suffices to show that
$\dim_{\Z(S)}(jSj) = n_1^2n_2^2$,
since $n_i^2  = \dim_{\Z(S)}(S_i)$.
The isomorphism
$\crp{F}K e_{(\theta, \crp{F})}
 \xrightarrow{\cdot e_{\theta}}
 \crp{F}(\theta)K e_{\theta}$ takes
$jSj$ onto $T:=(e_{\theta}e_{\eta}e_{\phi}\crp{F}(\theta)K
e_{\theta}e_{\eta}e_{\phi})^L$.
To compute the dimension of $T$ over $\crp{F}(\theta)=\crp{E}$, we
may assume that $\crp{E}$ is a splitting field of
$K$, $N$ and $L$.
Let $V$ be a $\crp{E}K$\nbd module affording $\theta$.
  Then $Ve_{\eta}\iso n_1W$ as $\crp{E}N$\nbd module,
  where $W$ affords $\eta$, and
  $Ve_{\eta}e_{\phi}\iso n_1n_2X$ as $\crp{E}L$\nbd module, where
  $X$ affords $\phi$.
  Thus
  $T    = (e_{\eta}e_{\phi}\crp{E}Ke_{\theta}e_{\eta}e_{\phi})^L
    \iso \enmo_{\crp{E}L}(Ve_{\eta}e_{\phi})
    \iso \mat_{n_1n_2}(\crp{E}) $.
  Therefore $\dim_{\crp{E}} (T) = (n_1n_2)^2$ as claimed.
\end{proof}
Now let $Z=\Z(S)$ and set $Z_0 = Z^H$, the subfield fixed by the
action of $H$.
Our assumption that there are crossed magic representations
$\sigma_i$ includes the assumption that the algebras $S_i$ are
obtained by scalar extension from central simple $Z_0$-algebras
$S_{10}$ and $S_{20}$, say.
The isomorphism of Lemma~\ref{l:s1tensors2} sends
$S_{10}\tensor_{Z_0} S_{20}$ onto a central simple
$Z_0$\nbd algebra $T_0$, such that
$jSj \iso T_0 \tensor_{Z_0}Z$.
Thus the algebra $jSj$ is obtained by scalar extension from $Z_0$,
and so is the algebra class of $jSj$ in the Brauer group of $\Z(S)$.
Of course, $S$ belongs to the same class as $jSj$.
However, it doesn't follow that $S$ itself is obtained by scalar
extension from a $Z_0$-algebra.
(There exist examples to the end that some algebra $S$
can not be obtained by scalar extension,
while $\mat_k(S)$ can, for some $k>1$~\citep[\S 14]{eilmac48}.)
Thus we make the following assumption:
\begin{hyp}\label{h:bconf2_scalar}
  There are central simple $Z_0$\nbd algebras
  $S_{10}\subseteq S_1$, $S_{20}\subseteq S_2$ and $S_0\subseteq S$,
  such that
  \[ S_1 \iso S_{10}\tensor_{Z_0} Z, \quad
     S_2 \iso S_{20}\tensor_{Z_0} Z, \quad
     S \iso S_0 \tensor_{Z_0} Z
     \quad \text{and}\quad
     S_{10}S_{20}\subseteq S_0.\]
  Moreover, there are magic crossed representations
  \[ \sigma_1\colon U/N\to S_1
     \quad\text{and}\quad
     \sigma_2\colon H/L\to S_2
  \]
  which are crossed over
  $S_{10}$ and $S_{20}$, respectively.
\end{hyp}
Thus if we view $S_{10}$ and $S_{20}$ as given, we search for a
central simple $Z_0$\nbd algebra $S_0$ that contains both
$S_{10}j$ and $S_{20}j$.
This is somewhat more than only to assume that $S$ is obtained
by scalar extension from a $Z_0$\nbd algebra.
\begin{remark}
  Suppose $S_{i0}\subseteq S_i$ and $\sigma_i$ are given.
  Then Hypothesis~\ref{h:bconf2_scalar}
  holds in each of the following cases:
  \begin{enums}
  \item $Z=Z_0$ (that is, $\theta$, $\eta$ and $\phi$
             are invariant in $H$).
  \item $S_{10}$ and $S_{20}$ are matrix rings over $Z_0$ and
       $S$ is a matrix ring over $Z$.
  \item $j=i$ or, equivalently, $n=n_1 n_2$.
  \item $n_1$, $n_2$ and $n$  are all coprime to
        $\abs{Z:Z_0}$.
  \end{enums}
  (In the last case the result follows from Proposition~\ref{p:dadeschmid}.)
\end{remark}
We are now able to state the main result of this section:
\begin{prop}\label{p:induc}
  In the situation of Lemma~\ref{l:bconf_comp}, assume
  Hypothesis~\ref{h:bconf2_scalar}.
  Then there is a magic crossed representation
  \[ \sigma\colon H/L \to S = (e_{\phi}\crp{F}Ke_{\theta}e_{\phi})^L\]
   such that
  \[ \iota(\sigma) = \iota(\sigma_1)\iota(\sigma_2).\]
  For this $\sigma$, we have
  $\sigma(h)j = \sigma_1(h)\sigma_2(h)$.
\end{prop}
\begin{proof}
Set $T_0= S_{10}S_{20}$ and $T= jSj$.
Define $\tau(h)= \sigma_1(h)\sigma_2(h)\in T$.
Then for every $t_0=s_{10}s_{20}\in T_0$ we have
\[ t_0^{\tau(h)}
    = s_{10}^{\sigma_1(h)}s_{20}^{\sigma_2(h)}
    = s_{10}^h s_{20}^h
    = t_0^h,\]
since $S_1j$ and $S_2j$ commute with each other.

As before, define $\eps(h)\in \Aut S$ by
$(s_0 z)^{\eps(h)}= s_0 z^{\eps(h)}$.
Then
\begin{align*}
 \tau(x)^{\eps(y)}\tau(y)
   &= \sigma_1(x)^{\eps(y)}\sigma_1(y) \sigma_2(x)^{\eps(y)}\sigma_2(y)
   \\
   &= \sigma_1(xy)\sigma_2(xy)=\tau(xy).
\end{align*}
Again, the first equality holds since $S_1j$ and $S_2j$ commute,
and the second follows since $S_{10}$, $S_{20}\subseteq S_0$ and
$\sigma_1$ and $\sigma_2$ are crossed over $S_{10}$ and
$S_{20}$, respectively.

By Lemma~\ref{l:crcoc}, a projective crossed
representation $\sigma\colon H/L \to S$ exists,
such that $s_0^h = s_0^{\sigma(h)}$ for all
  $s_0\in S_0$.
  In particular, this holds for elements of $T_0$.
  The idempotent $j= ji$ is
  $H$\nbd invariant by assumption, so that $\sigma(h)$ centralizes
  $j$ for every $h\in H$.
  Thus $j\sigma(h)= \sigma(h)j\in T$,
  and this element is  invertible in $T$.
  For $t_0\in T_0$ we have thus
  $t_0^{j\sigma(h)}= t_0^{\sigma(h)}= t_0^h = t_0^{\tau(h)}$.
  As $\C_T(T_0)= Zj$, it follows that
  $j\sigma(h) = \lambda_h \tau(h)$ for some
  $\lambda_h\in Z$.
  Therefore $\sigma$ is projectively equivalent with a
  representation.

From now on, assume that
  $\sigma\colon H/L\to S$ is such that
  $j\sigma(h)= \tau(h)=\sigma_1(h)\sigma_2(h)$.
  We want to show that
  $\iota(\sigma)= \iota(\sigma_1)\iota(\sigma_2)$.
  Choose $s_{i}\in S_{i0}$ with
  $\tr_{S_{i0}/Z_0}(s_i)=1$.
  Then
  \[\tr_{S_0/Z_0}(s_1s_2)= \tr_{T_0/Z_0}(s_1s_2)
         = \tr_{S_{10}/Z_0}(s_1)\tr_{S_{20}/Z_0}(s_2)=1.
  \]
  Let $\chi\in \compl[\Irr(G\mid e_{(\theta, \crp{F})})]$.
  Let $u\in U$ and $\alpha\in \crp{F}N$, and
  write $\alpha= \sum_{n\in N}\alpha_n n$
  with $\alpha_n\in \crp{F}$.
  Since $N\leq \ker \sigma_1$, we have
  \begin{align*}
    \chi^{\iota(\sigma_1)}(\alpha u)
      = \sum_{n\in N}\alpha_n \chi\big( s_1 \sigma_1(nu)^{-1} nu\big)
     &= \sum_{n\in N} \alpha_n \chi\big( s_1 \sigma_1(u)^{-1}nu\big)
      \\
     &=\chi\big( s_1 \sigma_1(u)^{-1}\alpha u\big).
  \end{align*}
  Using this with $\alpha= s_2\sigma_2(h)^{-1}$, where $h\in H$ is
  arbitrary, we get
  \begin{alignat*}{3}
   \chi^{\iota(\sigma_1)\iota(\sigma_2)}( h )
      &= \chi^{\iota(\sigma_1)}( s_2\sigma_2(h)^{-1} h ) && \\
      &= \chi ( s_1 \sigma_1(h)^{-1} s_2 \sigma_2(h)^{-1} h )
         & \quad &\text{ (see above)} \\
      &= \chi( s_1 s_2 \tau(h)^{-1} h ) &&\\
      &= \chi( s_1s_2 \sigma(h)^{-1} h )
         &\quad & \text{ (as $s_1s_2\in j Sj)$} \\
      &= \chi^{\iota(\sigma)}(h)
        &\quad & \text{ (as $\tr_{S_0/Z_0}(s_1s_2)=1$) }.
        \mbox{\qedhere}
  \end{alignat*}
\end{proof}

\section{Lifting magic representations}\label{sec:lifting}
In this section, $\crp{K}$ denotes a field complete
with respect to a discrete valuation
$\nu\colon \crp{K}\to \ints$, with valuation ring
$A$ and with a perfect residue class field $\crp{F}= A/\J(A)$
of characteristic $p>0$.
The letter $\pi$ denotes a prime element of $A$, so that
$\J(A)=A\pi $.

We need the following lemma which is probably well known.
\begin{lemma}\label{l:dvrcoh}
  Let $X$ be a group that acts on $A$ and
  $\alpha\in Z^2(X, 1+ A\pi )$.
  Then the order of the cohomology class of $\alpha$
  (in $H^2(X, 1+A\pi)$, and thus in $H^2(X, A^*)$)
  is a $p$\nbd number
  (or $\infty$, if\/ $\abs{X}$ is not finite).
\end{lemma}
\begin{proof}
  It suffices to show that if the cohomology class of $\alpha$ has $p'$\nbd order,
  then $\alpha$ is a coboundary.
  So suppose that $\alpha$ is a cocycle
  such that $\alpha^k$ is a coboundary for some $p'$\nbd number $k$,
  say.
  Thus $\alpha(x,y)^k= \lambda(x)^y\lambda(y)\lambda(xy)^{-1}$ for some
  $\lambda\colon  X \to 1+A\pi$.

  We will construct a sequence of maps
  $\mu_n\colon X \to 1+A\pi$ ($n=1,2,\dotsc $) such that
  \begin{align*}
    \alpha(x,y)\mu_n(xy) &\equiv \mu_n(x)^y\mu_n(y) \mod \pi^n
                         \quad \text{and}  \\
    \mu_{n+1}(x) &\equiv \mu_{n}(x) \mod \pi^n
  \end{align*}
  for all $n$.
  It follows that $\mu(x)= \lim_{n\to \infty}\mu_n(x)$ exists,
  and that
  \[\alpha(x,y)= \mu(x)^y\mu(y)\mu(xy)^{-1},
  \]
  as we want to show.

  Choose $a,b\in \ints$ with
  $ak+bp = 1$.
  We define $\mu_n$ recursively by
  \[ \mu_1(x)=1_A,
    \quad  \mu_{n+1}(x)= \lambda(x)^a\mu_n(x)^{bp}
    \quad (x\in X).
  \]
  We use induction to prove the above properties:
  By assumption, we have $\alpha(x,y)\equiv 1 \mod \pi$ for all $x,y\in X$.
  Also $\mu_2(x)=\lambda(x)^a\equiv 1 =\mu_1(x)\mod \pi$ since
  $\lambda(x)\in 1+A\pi$ by assumption.
  From $\mu_{n+1}(x)\equiv \mu_{n}(x)\mod \pi^{n}$ it follows
  $\mu_{n+1}(x)^p \equiv \mu_{n}(x)^p \mod \pi^{n+1}$ and thus
  \begin{align*}
    \mu_{n+2}(x) &= \lambda(x)^a\mu_{n+1}(x)^{bp} \\
              &\equiv  \lambda(x)^a\mu_{n}(x)^{bp} \mod \pi^{n+1}\\
              &= \mu_{n+1}(x).
  \end{align*}
  Assuming that
  $\alpha(x,y)\equiv \mu_n(x)^y\mu_n(y) \mu_n(xy)^{-1}\mod \pi^n$ by
  induction, we get
  \begin{align*}
  \alpha(x,y) &= \alpha(x,y)^{ak+bp}
            = \left(\lambda(x)^y \lambda(y)\lambda(xy)^{-1}\right)^a \alpha(x,y)^{bp} \\
            &\equiv \left(\lambda(x)^y \lambda(y) \lambda(xy)^{-1} \right)^a
                \left(\mu_n(x)^y \mu_n(y) \mu_n(xy)^{-1} \right)^{bp}
                \mod \pi^{n+1}\\
            &=
            \mu_{n+1}(x)^y\mu_{n+1}(y)\mu_{n+1}(xy)^{-1}.
  \end{align*}
  The proof is finished.
\end{proof}
For convenience, we now fix notation to be used in the rest
of this section.
\begin{hyp}\label{h:bconf2_dvr}
Let $A$ be a complete discrete valuation ring with quotient field
$\crp{K}$ of characteristic zero and residue class field
$A/A\pi =\crp{F}$ of characteristic $p>0$.
Suppose we are in the situation of
Hypothesis~\ref{h:bconf2} with $\crp{K}$ instead of $\crp{F}$.
Recall that $K\nteq G$ and $H\leq G$ with $G=KH$ and
$L=H\cap K$, and that $\theta\in \Irr K$ and $\phi\in \Irr L$ are such
that $\crp{K}(\theta)=\crp{K}(\phi)$, and for every
$h\in H$ there is $\gamma_h\in \Gal(\crp{K}(\theta)/\crp{K})$ such
that $\theta^{h\gamma_h}=\theta $ and $ \phi^{h\gamma_h}=\phi$.

Let $\crp{L}=\crp{K}(\theta)$
and let $B$ be the integral closure of $A$ in
$\crp{L}$.
It is well known that $\nu$ extends uniquely to a valuation
on $\crp{L}$ and that $B$ is the corresponding
valuation ring~\citep[Chap.~II, {\S}~2]{serreCL}.
Let $\crp{E}= B/\J(B)$ be the residue class field of $\crp{L}$.

In addition, we assume that $\theta$ and $\phi$ have
$p$\nbd defect zero.
\end{hyp}
It follows that
$e_{\theta}\in BK$ and $e_{\phi}\in BL$.
Moreover,  $\theta$ and $\phi$ vanish on elements of order divisible
by $p$~\citep[8.17]{isaCTdov}.
Thus the values of $\theta$ and $\phi$
are contained in $\rats(\eps)$, where
$\eps$ is a primitive $m$\nbd th root
of unity with $m$ not divisible by $p$.
It follows that $\crp{L}$ is unramified over
$\crp{K}$~\citep[Chap.~IV, {\S}~4]{serreCL}.
Thus $\abs{\crp{L}:\crp{K}}=\abs{\crp{E}:\crp{F}}$ and
$\J(B)=B \pi$. The canonical homomorphism
$\Gal(\crp{L}/\crp{K})\to \Gal(\crp{E}/\crp{F})$ is an
isomorphism.

The canonical epimorphism $\overline{\phantom{x}}\colon B \to\crp{E}$ extends
naturally to $BG$, and we denote this extension also by
$\overline{\phantom{x}}$.
Since we will work mostly with $AG$ and $\crp{F}G$, we emphasize
that we also use
the symbol $\overline{\phantom{x}}$ to denote the restriction
$\overline{\phantom{x}} \colon AG \to \crp{F}G$,
which is an epimorphism from $AG$ onto $\crp{F}G$.

Suppose there is a magic crossed representation
$\sigma\colon H/L\to(\overline{i}\crp{F}K \overline{i})^L$,
where
$i= \sum_{\gamma\in \Gal(\crp{L}/\crp{K})} e_{\theta}e_{\phi}$.
We wish to show that
$\sigma$ lifts to a crossed magic representation
$\widehat{\sigma}\colon H/L \to (i AK i)^L
\subseteq (i\crp{K} K i)^L$.

Before we do this, we observe the following:
\begin{lemma}\label{l:matrb}
  $(\overline{i}\crp{F}K \overline{i})^L
   \iso \mat_n(\crp{E})$ and
  $(i A K i)^L \iso \mat_n(B)$.
\end{lemma}
\begin{proof}
  The first isomorphism follows since
  $\Z((\overline{i}\crp{F}K \overline{i})^L)\iso \crp{E}$
  and $\crp{F}$ has characteristic $p>0$:
  Because then $\crp{F}K \overline{e_{(\theta, \crp{K})}}$ and
  $\crp{F}L \overline{i}\iso
   \crp{F}L \overline{e_{(\phi, \crp{K})}} $
  are matrix rings over
  $\crp{E}$, and thus
  $(\overline{i}\crp{F}K \overline{i})^L
   =
   \C_{\overline{i}\crp{F}K \overline{i}}(\crp{F}L \overline{i}) $
  must be a matrix ring, too.

  Since $A$ is complete, we may lift idempotents,
  and thus the matrix units,
  to $(i A K i)^L$~\citep[Proposition~21.34]{lamFC}.
  The second isomorphism follows.
\end{proof}
Next we show that
Lemma~\ref{l:crcoc} extends to valuation rings.
Set $\Sigma = (iAK i)^L$ and $S= (i\crp{K}Ki)^L$.
The
isomorphism $\Z(S)\iso \crp{L}$ of Lemma~\ref{l:centerisos}
makes $S$ into an $\crp{L}$\nbd algebra.
Restriction to $B$ yields an isomorphism $B\iso \Z(\Sigma)$,
so we may view $\Sigma$ as
a $B$\nbd algebra.
These isomorphisms respect the action of $H$.
    Let $\crp{L}_0 =\crp{L}^H$, the subfield of elements fixed by
  $H$,
  and set $B_0 = B\cap \crp{L}_0 (= B^H)$.

  By Lemma~\ref{l:matrb}, $\Sigma$ contains a set of matrix units,
  $E$ (say).
  Let $\Sigma_0$ be the $B_0$\nbd subalgebra generated by
  $E$, so that $\Sigma_0\iso\mat_n(B_0)$
  and $\Sigma \iso \Sigma_0 \tensor_{B_0} B$.
  Similarly, let $S_0$ be the
  $\crp{L}_0$\nbd subalgebra of $S$ generated by $E$.
  Then clearly $S_0 \iso \Sigma_0\tensor_{B_0}\crp{L}_0$.

  The identification of $S$ with a matrix ring over $\crp{L}$
  yields an action of $\Gal(\crp{L}/\crp{K})$ on $S$;
  since $H$ acts on $\crp{L}$, we get an action of $H$ on $S$
  which we denote by $\eps$.
  Thus, for $s_0\in S_0$ and $z\in \Z(S)$,
  we have $(s_0 z)^{\eps(h)}= s_0 z^h$.
  This action maps $\Sigma$ into $\Sigma$.

  Since $\Sigma$ is a matrix ring over the local ring $\Z(\Sigma)\iso B$,
  every automorphism of $\Sigma$ centralizing $\Z(\Sigma)$
  is inner~\citep[Chap.~2, Theorem~4.8]{ntRFG}.
  This applies to $s\mapsto s^{\eps(h)^{-1}h}$.
  Thus there is, for every $h\in H$, an element
  $\sigma(h)\in \Sigma^*$ such that
  $s_0^h = s_0^{\sigma(h)}$ for all $s_0\in \Sigma_0$
  (and then in fact for all $s_0\in S_0$).
  This yields a
  crossed projective representation
  $\sigma\colon  H/L \to \Sigma^* $.
  We have thus proved:
\begin{lemma}\label{l:dvrprcr}
  In the situation of Hypothesis~\ref{h:bconf2_dvr}, there
  is a crossed projective representation
  $\sigma\colon H/L \to (iAK i)^L=\Sigma$ such that
  $s^h = s^{\sigma(Lh)}$ for all $s\in \Sigma_0$, where
  $\mat_n(B^H)\iso\Sigma_0\subseteq (iAKi)^L$.
\end{lemma}
\begin{thm}\label{t:dvr2}
  In the situation of Hypothesis~\ref{h:bconf2_dvr},
  suppose that
  there is a magic crossed representation
  $\sigma\colon  H/L \to (\overline{i}\crp{F}K \overline{i})^L$.
  Then for every $p'$\nbd subgroup $V/L\leq H/L$ there is a
  magic crossed representation $\widehat{\sigma}\colon V/L \to (iAKi)^L $
  lifting $\sigma_{V/L}$
  (that is,
   $\overline{\widehat{\sigma}(x)}= \sigma(x)$
   for $x\in V/L$).
  If $n=(\theta_L, \phi)\not\equiv 0\mod p$, then there is
  a magic crossed representation $\widehat{\sigma}\colon H/L\to (iAKi)^L $
  lifting $\sigma$.
\end{thm}
\begin{proof}
  Let
  $\widehat{\sigma}\colon  H/L \to \Sigma =(iAKi)^L$
  be a
  crossed projective representation, which exists
  by Lemma~\ref{l:dvrprcr}.

  Since $\widehat{\sigma}(h)\in \Sigma^{*}$,
  reduction modulo $\pi$ yields a
  crossed projective representation
  \[ h \mapsto
     \overline{\widehat{\sigma}(h)}
     \in
     \overline{\Sigma}
     = (\overline{i}\crp{F}K \overline{i})^L.
  \]
  Clearly, $t^{\overline{\widehat{\sigma}(h)}} = t^h$ for
  $t\in \overline{\Sigma}_0\iso \mat_n(\crp{E}^H)$.
  After multiplying $\widehat{\sigma}$ with a suitable factor from
  $\Z(\Sigma)$, we may assume that
  $\overline{\widehat{\sigma}(h)}=\sigma(h)$ for $h\in H$.
  Let $\alpha\in Z^2(H/L, B^*)$ be the cocycle associated with
  $\widehat{\sigma}$. Then $\alpha$ has values in $1+B\pi$, since
  $\sigma$ is multiplicative. By Lemma~\ref{l:dvrcoh}, the
  cohomology class of $\alpha$ has $p$\nbd order.
  In particular, $\alpha_{V/L}\sim 1$ for any $p'$\nbd group $V/L$.
  Thus $\widehat{\sigma}_{V/L}$ is projectively equivalent with a
  crossed representation.
  If $n\not\equiv 0 \mod p$, then it follows $\alpha \sim 1$,
  since the class of $\alpha$ has order dividing $n$.
  The proof is finished.
\end{proof}

\section{Reducing magic representations modulo a
prime}\label{sec:reduc}
Assume Hypothesis~\ref{h:bconf2_dvr}.
As before, let
$i = \sum_{\gamma\in \Gal(\crp{L}/\crp{K})}
(e_{\theta}e_{\phi})^{\gamma}$ and
$S= (i\crp{K}Ki)^L$.
Our aim here is to show that if there is a magic representation
$\sigma\colon H/L \to S$, then not only
$\crp{K}G e_{(\theta,\crp{K})}$ and $\crp{K}He_{(\phi, \crp{K})}$
are Morita equivalent, but also
$A Ge_{(\theta,\crp{K})}$ and
$A He_{(\phi, \crp{K})}$.

There is a quite general result of Brou\'{e} of this
kind~\citep{broue90b}, but verifying the premises of Brou\'{e}'s result
is nearly the same amount of work as proving the desired result
directly.

\begin{lemma}
  $AKe_{(\theta,\crp{K})} = A K i A K$.
\end{lemma}
\begin{proof}
  As $\crp{F}K \overline{e_{(\theta,\crp{K})}}$ is  simple, we
  have $\crp{F}K \overline{e_{(\theta,\crp{K})}}=
       \crp{F}K \overline{i} \crp{F}K$.
  Thus
  \[AKe_{(\theta, \crp{K} ) } = A K i A K + \pi AKe_{(\theta, \crp{K} ) }.
  \]
  By Nakayama's lemma, the result follows.
\end{proof}
It follows that $AGe_{(\theta, \crp{K} ) }$ and $i AG i$ are
Morita equivalent.
We now assume that a magic crossed representation
$\sigma\colon H/L \to S\iso \mat_n(\crp{L})$
exists.
We know that then $i \crp{K}G i \iso \mat_n(\crp{K}He_{(\phi,\crp{K})})$, and we want to
show that the same is true if we replace $\crp{K}$ by $A$.
We have seen in the last section that
$\Sigma = (iAKi)^L \iso \mat_n(B)$.
Choose $\Sigma_0\subseteq \Sigma$ with
$\Sigma_0\iso\mat_n(B^H)$ and let
$S_0 = \crp{K}\Sigma_0 \iso \crp{K}\tensor_A \Sigma_0$.
Then
$iAGi \iso \mat_n(\Gamma)$, where $\Gamma= \C_{iAGi}(\Sigma_0)$.
It is clear that $\Gamma = AG \cap C$, where
$C=\C_{i\crp{K}Gi}(S_0)$.
We have to show that the isomorphism
of Theorem~\ref{t:c-isosemi}
sends $AHe_{(\phi, \crp{K})}$ onto $\Gamma$.
We first show
   that $ \sigma(H) \subseteq \Sigma = (iAKi)^L$.
\begin{lemma}\label{l:sigminSigm}
  In the situation of Hypothesis~\ref{h:bconf2_dvr}, suppose that
  there is a magic crossed representation
  $\sigma\colon H/L \to S = (i\crp{K}K i)^L$.
  Then $\sigma(H/L)\subseteq \Sigma=(iAK i)^L$.
\end{lemma}
\begin{proof}
   By Lemma~\ref{l:dvrprcr},
   there also exists a projective crossed representation
   $\widehat{\sigma}\colon H/L\to \Sigma$.
   For $h\in H$, there is $\lambda\in \Z(S)\iso\crp{L}$ such that
   $\widehat{\sigma}(h)=\lambda\sigma(h)$, and thus
  $\lambda\sigma(h)\in \Sigma^{*}$.
  This means that also
  $(\lambda\sigma(h))^{-1} =\lambda^{-1}\sigma(h^{-1})\in \Sigma^{*}$.
  As $\lambda$ or $\lambda^{-1}$ is in $\Z(\Sigma)\iso B$, it follows that
  $\sigma(h^{-1})$ or $\sigma(h)$ is in $\Sigma$.
  But as $\sigma$ is a crossed representation and $h$ has finite order,
  both are in $\Sigma$. Thus $\sigma(H) \subseteq \Sigma$ as
  claimed.
\end{proof}
\begin{remark*}
  Completeness of $\crp{K}$ was not used in the previous proof,
  and in fact the lemma is true for any ring $A$ integrally closed
  in its quotient field, and such that $i\in AK$.
  This follows because such a ring is the intersection of all the
  valuation rings containing it.
\end{remark*}
\begin{lemma}
  Keep the notation above and assume that
  $\sigma\colon H/L \to S$ is magic.
  Then the homomorphism
  $\kappa$ of Theorem~\ref{t:c-isosemi} maps
  $AHe_{\phi}$ onto $ \Gamma$.
\end{lemma}
\begin{proof}
  From Lemma~\ref{l:sigminSigm} it follows that
  $c_h = h\sigma(Lh)^{-1}\in C\cap AG= \Gamma$.
  Therefore $(AHe_{(\phi,\crp{K})})\kappa \subseteq \Gamma$.
  As in the proof of Theorem~\ref{t:c-isosemi}, we see that
  $\Gamma = \bigoplus_{t\in T}\Gamma_1 c_t$, where
  $\Gamma_1 = \Gamma\cap AK$ and $T$ is a set of representative
  of the cosets of $L$ in $H$.
  Thus it suffices to show that
  $(ALe_{(\phi,\crp{K})}){\kappa}= \Gamma_1$.
  We know that $(ALe_{(\phi,\crp{K})}){\kappa}= ALi$ and that
  $\Gamma_1 = \crp{K}Li \cap iAKi$.
  Clearly $ALi\subseteq \Gamma_1$.
  But since
  $ALi \iso ALe_{(\phi,\crp{K})}\iso \mat_n(B)$, the subring
  $ALi$ must be a maximal
  $A$\nbd order of
  $\crp{K}Li\iso\mat_n(B)$~\citep[Theorem~8.7]{reinerMO}.
  Since $\Gamma_1\subseteq iAKi$ is an $A$\nbd order,
  the proof now follows.
\end{proof}
From what we have done so far it follows that,
if a magic crossed representation
$\sigma \colon H/L \to S=(i\crp{K}K i)^L$ exists
in the situation of Hypothesis~\ref{h:bconf2_dvr},
then
$AHe_{\phi}$ and $AGe_{(\theta, \crp{K} ) }$ are
Morita equivalent,
and the equivalence induces the character bijection of
Theorem~\ref{t:semicorrchars}.
The equivalence can be described more concretely:
First, choose an primitive idempotent
$j\in \Sigma = (iAKi)^L$.
Then we have $jAGj = jiAGij \iso \Gamma \iso AHe_{\phi}$, where
an isomorphism from $AHe_{\phi}$ onto $jAGj$ is induced by the map
sending $h\in H$ to
$jh\sigma(Lh)^{-1}= h\sigma(Lh)^{-1}j = j h\sigma(Lh)^{-1} j$.
Also we have
$ i = 1_S \in \Sigma j \Sigma $
and $e_{(\theta, \crp{K} ) } \in AG i AG = AGe_{(\theta, \crp{K} ) }$, so that
$AG j AG = AGe_{(\theta, \crp{K} ) }$.
The idempotent $j$ is thus full in $AGe_{(\theta, \crp{K} ) }$ and we have
a Morita\index{Morita equivalence} equivalence between $AGe_{(\theta, \crp{K} ) }$ and
$jAGj=jAGe_{(\theta, \crp{K} ) }j$ sending
an $AGe_{(\theta, \crp{K} ) }$-module $V$ to $Vj$ and an
$jAGj$\nbd module
$U$ to $U\tensor_{jAGj} jAG$~\citep[Example~18.30]{lamMR}.
Since $jAGj\iso AHe_{\phi}$, this gives also an Morita\index{Morita equivalence} equivalence
between $AGe_{(\theta, \crp{K} ) }$ and $AHe_{\phi}$.
We now have proved:
\begin{thm}\label{t:dvr1}
  Assume Hypothesis~\ref{h:bconf2_dvr} and let $\sigma\colon H/L \to S$
  be a magic representation.
  Then there is an idempotent $j\in (iAKi)^L= \Sigma$ such that
  $AHe_{(\phi,\crp{K})}\iso j AG j$ via the map defined by
  $h\mapsto j h\sigma(Lh)^{-1}$,
  and $AG j AG = AGe_{(\theta, \crp{K} ) }$.
  The rings $AHe_{(\phi,\crp{K})}$ and $AGe_{(\theta, \crp{K} ) }$ are
  Morita equivalent.
\end{thm}
\begin{remark*}
 The Morita equivalence is graded in the sense
 of~\citep{marcusGGA, marcus08}.
\end{remark*}

\section{Above the Glauberman correspondence}
\label{sec:aboveglaub}
In this section, we need some notation and results from the theory
of $G$-algebras, which we review shortly.
Let $A$ be an algebra, and suppose the group $G$ acts on $A$.
(Shortly, $A$ is a $G$\nbd algebra.)
Remember that for subgroups $Q\leq P\leq G$ we have a trace map
$\T_Q^P\colon A^Q \to A^P$ given by
$\T_Q^P(a)= \sum_{x\in [P:Q]}a^x$
(where $[P:Q]$ denotes a set of representatives
 of the left cosets of $Q$ in $P$),
and that
$A_Q^P:=\T_Q^P(A^Q)$ is an ideal of $A^P$~\citep[{\S}~11]{thevGA}.

Now let $A$ be an algebra over a field $\crp{E}$ of characteristic
$p>0$.
We write $A(P)$ for the factor algebra
$A^P/ (A^P_{<P})$, where
$A^P_{<P}$ denotes the ideal
$\sum_{Q<P} A_Q^P$, the sum running over proper subgroups of
$P$.
The canonical homomorphism $\br_P=\br_P^A\colon A^P\to A(P)$ is
called the Brauer homomorphism.
If $A(P)\neq 0$, then $P$ is a $p$\nbd group.

In the special case where $A=\crp{E}K$ and $P$ acts on the group
$K$, we may (and do) identify $A(P)$ with $\crp{E}\C_K(P)$; the Brauer
homomorphism becomes the usual projection
$(\crp{E}K)^P\to \crp{E}\C_K(P)$.

If $e\in A^G$ is a primitive idempotent, then
the minimal subgroups among the subgroups $D\leq G$ for which
$e\in A_D^G$, are conjugate in $G$ and are called
\emph{defect groups} of $e$.
Defect groups are also characterized as maximal subgroups subject
to $\br_D(e)\neq 0$~\citep[{\S}~18]{thevGA}.

We describe now the situation we study in this section.
(It is the same situation as studied by Dade~\citep{dade80}.)
Let $G$ be a finite group with normal subgroups
$K \leq M$ such that $M/K$ is a $p$\nbd group.
Let $\theta\in \Irr K$ be invariant in $M$ and
semi-invariant in $G$ and suppose that
$\theta$ has $p$-defect zero.
Observe that then the coefficients of the central idempotent
$e_{\theta}\in \rats(\theta)G$ are contained in the valuation ring
$\ints_p[\theta]$ (which is unramified over $\ints_p$).
Let $\crp{E}=\ints_p[\theta]/p\ints_p[\theta]$ be the residue
class field.
Write $\overline{e_{\theta}}$ for the image of
$e_{\theta}$ in $\crp{E}K$ under the homomorphism
$\ints_p[\theta]K \to\crp{E}K$.
Then
$\overline{e_{\theta}}\in \Z(\crp{E}K)$
is a block idempotent of $\crp{E}K$,
and we have
$\overline{e_{\theta}} = \T_1^K(a)$ for some
$a\in \crp{E}K\overline{e_{\theta}}$, since $\theta$ has
$p$\nbd defect zero.

One may view $\crp{E}K$ as an $M$\nbd algebra.
Let $P\leq M$ be a defect group of the idempotent
$\overline{e_{\theta}} \in (\crp{E}K)^M$.
Then $M=KP$ and $K\cap P=1$.
Set $H=\N_G(P)$. Since $\theta$ is semi-invariant in $G$ and
Galois conjugate blocks have the same defect groups, it follows
from the Frattini argument that $G=HK$.
Set $L=H\cap K$, so that $L= \C_K(P)$.

Let $\beta=\br_P^{\crp{E}K}\colon (\crp{E}K)^P\to \crp{E}L$ be the
Brauer homomorphism.
It induces a bijection between the primitive idempotents
in $(\crp{E}K)^M = \Z(\crp{E}K)^P$ with defect $P$, and the
primitive idempotents of $(\crp{E}L)^{\N_M(P)}$ with defect $P$.
Since $\N_M(P)=LP$, the last idempotents are the block idempotents of
$L$ with defect group $1$.

In particular, to $\theta$ corresponds a character
$\phi\in \Irr L$ of defect zero.
(When $K$ is a $p'$\nbd group, then
 this is the Glauberman correspondent
 of $\theta$.
 See also~\citep[{\S}5.12]{ntRFG}.)
The correspondence commutes with Galois automorphisms.
Thus we are in the situation of
Hypothesis~\ref{h:bconf2_dvr}
with $\crp{K}=\rats_p$, the field of
$p$\nbd adic numbers, and
$\crp{F}=\GF{p}$, the prime field with $p$ elements.
Let $e= \overline{e_{ (\theta,\rats_p) }}$ and
$f=\overline{ e_{ (\phi , \rats_p) } }$ be the central primitive
idempotents of $\crp{F}_p K$ and $\crp{F}_pL$ corresponding to
the blocks of $\theta$ and $\phi$ over the prime field $\crp{F}_p$.
(Remember that
$\Gal( \rats_p(\theta)/\rats_p)\iso \Gal(\crp{E}/\GF{p})$
canonically in this situation.)
In this section, we set
\[ i  = \sum_{\gamma\in \Gal(\crp{E} /\GF{p})}
        (\overline{e_{{\theta}}e_{{\phi}}})^{\gamma}.
  \]

\begin{thm}\label{t:glaubmag}
  In the situation just described, there is a magic crossed
  representation $\sigma\colon H/L\to
  (i\crp{F}_pK i)^L$.
  Moreover, $\sigma$ can be chosen such that
  $\br_P(\sigma(h))= f$ for all
  $h\in \C_G(P)$.
\end{thm}
We list some corollaries.
\begin{cor}\label{c:glaubbr}
  $\crp{F}_p Ge$ and\/ $\crp{F}_p Hf$ are ($H/L$-graded) Morita
  equivalent.
  Any block ideal of\/ $\crp{F}_p Ge$ is Morita equivalent to
  its Brauer correspondent with respect to $P$.
\end{cor}
\begin{proof}
  The first assertion follows from Corollary~\ref{c:corrsemi}.

   Now set $S= (i\crp{F}_p K i)^L$
  and remember that the magic crossed representation $\sigma$
  (with respect to $\mat_n(\crp{F}_p)\iso S_0 \subseteq S$)
  yields an
  isomorphism
  \[ \kappa\colon \crp{F}_p H f \to
      \C_{i\crp{F}_p G i}(S_0),
      \quad
      h \mapsto h \sigma(h)^{-1}.
    \]
  This restricts to an isomorphism
  from $\Z(\crp{F}_p H f)$ to
  $\Z(i\crp{F}_p G i )$.
  Furthermore, the map $z\mapsto zi$ is an isomorphism
  between $\Z(\crp{F}_p Ge)$ and $\Z(i\crp{F}_p G
  i)$.
  \[ \begin{CD}
      \Z(\crp{F}_p G e)
        @>{\cdot i}>>
       \Z(i\crp{F}_p G i)
        @<\kappa <<
       \Z(\crp{F}_p H f)
     \end{CD}.
  \]
  These isomorphisms yield a bijection between the
  block idempotents of $\crp{F}_p G e$ and $\crp{F}_p H f$.
  Let $b\in \Z(\crp{F}_p G e)$ and $c\in \Z(\crp{F}_p H f)$ be
  block idempotents that correspond under these isomorphisms, that
  is, we have $bi = c{\kappa}$.
  The Morita equivalence that
  comes from the magic crossed representation $\sigma$ then yields
  a Morita equivalence between the blocks
  $\crp{F}_p G b$ and $\crp{F}_p H c$.
  Thus we need to show that
  $c$ is the Brauer correspondent of $b$
  for an appropriate choice of $\sigma$.
  An appropriate choice is any $\sigma$ such
  that $\br_P(\sigma(h))= f$ for $h\in \C_G(P)$.

  Write $c= \sum_{h\in H} c_h h$ and observe that
  $c_h= 0$ for
  $h \notin \C_H(P)$~\citep[Theorem~5.2.8(ii)]{ntRFG},
  since $P$ is a normal
  $p$\nbd subgroup of $H$.
  We compute
  \begin{align*}
    \br_P(b)
      = \br_P(be)
      = \br_P(b)f
     &= \br_P(bi)
      = \br_P( c\kappa )
      \\
     &= \br_P\left( \sum_{h\in \C_G(P)} c_h h\sigma(h)^{-1} \right)
     \\
     &= \sum_{h\in \C_G(P)} c_h h \br_P(\sigma(h)^{-1})
     \\
     &= \sum_{h\in \C_G(P)} c_h h \cdot f
      = c,
  \end{align*}
  as was to be shown.
\end{proof}
\begin{cor}\label{c:glaub}
  There is a bijection $\iota$ between
  $\Irr(G\mid e_{ (\theta, \rats_p) })$ and
  $\Irr(H\mid e_{ (\phi, \rats_p) })$
  with the properties given in Theorem~\ref{t:corr}.
  In particular, the bijection $\iota$ commutes with field automorphisms over $\rats_p$
  and preserves Schur indices over $\rats_p$.
  Corresponding characters lie in Brauer corresponding blocks
  (with respect to $P$).
\end{cor}
This corollary is a slight generalization of one of the main results of Turull's
paper~\citep{turull08c}. In fact, this section gives a somewhat
different proof of Turull's result, but not completely
independent.
In particular, we need a fact about endopermutation modules
depending on Dade's classification of
these objects~\citep{dade78a, dade78b}.
This fact was announced by Dade~\citep{dade80}, but without proof.
A rather sketchy proof was published by Puig~\citep{puig86}.
We use the complete exposition with detailed proofs given by
Turull~\citep[Section~3]{turull08c}.

Our aim in this section is to show that Turull's result can be
derived naturally from the theory developed here.
We do not need Turull's theory of the Brauer-Clifford group~\citep{turull09}.
Also note that Turull assumes that $K$ is a
$p'$\nbd group,
and that Properties~\ref{i:c_res} and~\ref{i:c_ind}
in Theorem~\ref{t:corr}  and the connection with the
Brauer correspondence give
additional information.
\begin{proof}[Proof of Corollary~\ref{c:glaub}]
  We will show below (Lemma~\ref{l:sdade}) that
  \[\dim_{\crp{E}}(i\crp{F}_p Ki)^L\equiv 1\mod p.
  \]
  Thus $n=(\theta_L, \phi)\equiv \pm 1\mod p$.
  (Of course, this is a well known property of
   the Glauberman correspondence, cf.~\citep[\S 5.12]{ntRFG}.)
  By Theorem~\ref{t:dvr2}, the magic crossed representation $\sigma$ lifts
  to a magic representation in characteristic zero.
  The result follows from Theorem~\ref{t:semicorrchars}.
\end{proof}

To prove Theorem~\ref{t:glaubmag}, we need some facts about
endopermutation modules and Dade $P$\nbd algebras.

For the moment, let $P$ be a $p$-group and $\crp{E}$ an arbitrary field of characteristic $p$.
A permutation $\crp{E}P$\nbd module is a module with an
$\crp{E}$\nbd basis that is permuted by $P$.
(Such a basis is called \emph{$P$\nbd stable}.)
A permutation $P$\nbd algebra is
a $P$\nbd algebra $A$
that has a $P$\nbd stable basis.
An $\crp{E}P$\nbd module $V$ such that
$\enmo_{\crp{E}}V$ is a permutation $P$\nbd algebra, is called
an endopermutation module.
A Dade $P$-algebra over $\crp{E}$
is a permutation $P$\nbd algebra $S$ that is central simple over $\crp{E}$,
and such that $S(P)\neq 0$.
Equivalently, the central simple algebra $S$
has a $P$\nbd stable basis $B$ containing $1_S$.
This means that $S$ viewed as an $\crp{E}P$-module is a permutation
module containing the trivial module $\crp{E}$ as direct summand.

We need some properties of permutation modules and
Dade $P$\nbd algebras.
First some elementary facts:
\begin{lemma}\label{l:permp_basics}\hfill
  \begin{enums}
  \item Every direct summand of a permutation
        $P$\nbd module is itself a permutation $P$\nbd module.
  \item If $A$ and $B$ are permutation $P$\nbd algebras, then
        \[(A\tensor_{\crp{E}}B)(P)\iso A(P)\tensor_{\crp{E}}B(P)
        \]
        canonically.
  \item  If $S$ is a Dade
         $P$\nbd algebra
         and $S\iso\mat_n(\crp{E})$, then
         $S(P)$ is also a matrix ring over $\crp{E}$.
  \item \label{i:simplepalg} If $P$ acts on $\mat_n(\crp{E})$, where $\crp{E}$ is a
        perfect field, then there is a unique group homomorphism
        $\sigma\colon P\to \GL_n(\crp{E})$ such that
        $s^p = s^{\sigma(p)}$ for all $s\in\mat_n(\crp{E})$.
  \end{enums}
\end{lemma}
\begin{proof}
  \citep[Corollary~27.2, Proposition~28.3, Theorem~28.6(a), Corollary~21.4]{thevGA}.
\end{proof}
The proofs of the next  results depend on Dade's classification of
endopermutation modules for
abelian $p$-groups~\citep{dade78a, dade78b}.
\begin{lemma}\label{l:turull33}
  Let $P$ be an abelian $p$\nbd group, let
  $\crp{F}\leq \crp{E}$ be finite fields
  and
  $V$ an endopermutation $\crp{E}P$\nbd module.
  Then there is an
        endopermutation
        $\crp{F}P$\nbd module $V_0$ such that
        $V\iso V_0\tensor_{\crp{F}}\crp{E}$.
\end{lemma}
\begin{proof}
  \citep[Theorem~3.3]{turull08c}.
\end{proof}
\begin{cor}\label{c:turull34}
   Let $P$ be an abelian $p$\nbd group, let
  $\crp{F}\leq \crp{E}$ be finite fields
  and
  $V$ an endopermutation $\crp{E}P$\nbd module.
  Write
  $S=\enmo_{\crp{E}}V \subseteq T= \enmo_{\crp{F}} V$.
  Then $S$ and $T$ have $P$-stable bases,
  $S(P)$ embeds naturally in $T(P)$ and
  $\C_T(S)\iso \C_{T(P)}(S(P))\iso \crp{E}$ naturally.
\end{cor}
  That $T$ has a $P$-stable basis means that $V$ is an
  endopermutation module for $P$ over $\crp{F}$, this is
  Corollary~3.4 in~\citep{turull08c}.
\begin{proof}[Proof of Corollary~\ref{c:turull34}]
  By Lemma~\ref{l:turull33}, there is
  an endopermutation module $V_0\leq V_{\crp{F}P}$ such that
  $V\iso V_0\tensor_{\crp{F}}\crp{E}$.
  Set
  \[ S_0 = \{ s\in T \mid V_0s\subseteq V_0\}
       \iso \enmo_{\crp{F}}V_0.
  \]
  Since $V_0$ is an endopermutation module,
  $S_0$ has a $P$-stable basis, $B$.

  Set $Y=\enmo_{\crp{F}}\crp{E}$.
  Then $T\iso S_0 \tensor_{\crp{F}} Y$ as $\crp{F}$\nbd algebra.
  If we let act $P$ trivially on $Y$, this is actually an
  isomorphism of $P$\nbd algebras.
  Taking any $\crp{F}$\nbd basis $C$ of $Y$, we see that
  $\{ b\tensor c\mid b\in B, c\in C\}$ is a
  $P$\nbd stable basis of $T$.

  It is clear that
  $\C_T(S)= \enmo_S V \iso \crp{E}$.
  Let $\lambda(\crp{E})=\enmo_{\crp{E}}\crp{E}\subseteq Y$ and
  note that $\lambda(\crp{E})\iso \crp{E}$ via multiplication.
  We have $\C_{Y}(\lambda(\crp{E}))=\lambda(\crp{E})$.
  The isomorphism $T\iso S_0\tensor Y$ identifies
  $S$ with $ S_0\tensor \lambda(\crp{E})$.

  As $Y(P)= Y$ and $\crp{E}(P)=\crp{E}$, it now follows that
  $T(P)\iso S_0(P)\tensor_{\crp{F}}Y$ and
  $S(P)\iso S_0(P)\tensor_{\crp{F}}\crp{E}$ canonically.
  Thus also $\C_{T(P)}(S(P))\iso \crp{E}$.
\end{proof}
If $S$ is a Dade $P$-algebra, then,
by Lemma~\ref{l:permp_basics}, Part~\ref{i:simplepalg},
there is a unique group homomorphism
$\sigma\colon P\to S$ inducing the action of $P$ on $S$.
Thus the notation  $\N_{S^*}(P)$ is unambiguous, although strictly
speaking we should write $\N_{S^*}(\sigma(P))$.
\begin{lemma}\label{l:endoextsplit}
  Let $P$ be an abelian $p$-group, let $S$ be a Dade $P$-algebra and let
  $\br_P\colon S^P \to S(P)$ be the Brauer homomorphism.
  Then there is a group homomorphism
  $\phi\colon \N_{S^*}(P)\to S(P)^*$ such that
  $\phi$ extends $\br_P$ and
  $\br_P(c)^{\phi(s)}=\br_P(c^s)$ for all
  $s\in \N_{S^*}(P)$ and $c\in S^P$.
\end{lemma}
\begin{proof}
  \citep{puig86} or \citep[Theorem~3.11]{turull08c}.
\end{proof}
Set
\[ A= \{a \in \Aut S \mid \sigma(P)^a = \sigma(P)\}.\]
Here $\Aut S$ denotes the set of all ring automorphisms of $S$,
not just the $\crp{E}$-algebra automorphisms.
Since $A$ stabilizes $P$, it follows that
$A$ acts naturally on $S(P)$.
Lemma~\ref{l:endoextsplit} can be strengthened:
\begin{lemma}\label{l:endoextainv}
  If\/ $\crp{E}=\Z(S)$ is finite in the situation of Lemma~\ref{l:endoextsplit},
  the homomorphism $\phi\colon \N_{S^*}(P)\to S(P)$
  can be chosen such that
  $\phi(s^a)=\phi(s)^a$ for all $s\in \N_{S^*}(P)$ and
  $a\in A$.
\end{lemma}
\begin{proof}
  Let $V_S$ be a simple $S$-module.
  Then $V$ is an endopermutation module for $P$ over $\crp{E}$.
  It is still an endopermutation module when viewed as a module over
  the prime field
  $\GF{p}$, by Corollary~\ref{c:turull34}.
  Let $T= \enmo_{\GF{p}} V$ and view $S\iso\enmo_{\crp{E}} V$ as subset of $T$.
  Then $\C_T(S)=\crp{E}$ and $\C_{T}(\crp{E})=S$.

  The inclusion $S^P\subseteq T^P$ induces an injection
  $S(P)\into T(P)$
  by Corollary~\ref{c:turull34}, and
  $\C_{T(P)}(S(P))=\crp{E}$.

  By Lemma~\ref{l:endoextsplit} applied to $T$,
  there is a homomorphism
  $\phi\colon \N_{T^*}(P)\to T(P)$.
  We claim that $\phi(\N_{S^*}(P))\subseteq S(P)$.
  Let $s\in \N_{S^*}(P)$ and $z\in \Z(S)$.
  Then $\br_P(z)^{\phi(s)}=\br_P(z^s)=\br_P(z)$, so
  $\phi(s)$ centralizes $\br_P(\Z(S))\iso\Z(S)$.
  Since $\C_{T(P)}(\Z(S))=S(P)$, the claim follows.

  Now let $a\in A$. By the Skolem-Noether
  theorem~\citep[3.14]{farbNA}, there is $t\in T^*$ such that
  $s^a = s^t$ for all $s\in S$.
  Since $P^a=P$, it follows that $t\in \N_{T^*}(P)$.
  Thus for $s\in \N_{S^*}(P)$,
  \[\phi(s^a)= \phi(s^t)= \phi(s)^{\phi(t)}.\]
  Now $s\in \N_{S^*}(P)$ given,
  there is $c\in S^P\subseteq T^P$
  such that $\phi(s)= \br_P(c)$,
  as $\br_P\colon S^P \to S(P)$ is surjective.
  Then by Lemma~\ref{l:endoextsplit} applied to $T$ we have
  \[ \phi(s)^{\phi(t)} = \br_P(c)^{\phi(t)} = \br_P(c^{t})
      = \br_P(c^a) = \br_P(c)^a = \phi(s)^a,\]
  where the second last equation follows from the definition of the
  action of $A$ on $S(P)$.
\end{proof}

We work now in the situation of Theorem~\ref{t:glaubmag}.
We use the notation introduced at the beginning of the section
before the statement of Theorem~\ref{t:glaubmag}.
We set $S=(i\crp{F}_p Ki)^L$ and $Z=\Z(S)\iso \crp{E}$.
We begin with a lemma.
\begin{lemma}\label{l:sdade}
  $S$ is a Dade $P$-algebra with $S(P)\iso \crp{E}$.
  (Thus $\dim_{\crp{E}}S \equiv 1\mod p$.)
  We may identify
  $\br_P^S$ with the restriction of
  $\br_P^{\crp{F}K}$ to $S$.
\end{lemma}
\begin{proof}
  $\crp{F}K$ has a basis that is permuted by $P$.
  The same is thus true for the direct summand $i\crp{F}Ki$.
  By assumption, $P$ centralizes $Z\iso\Z(\crp{F}Ke)$.
  Let $B$ be some $Z$\nbd basis of $\crp{F}Lf$.
  From $i\crp{F}Ki \iso S\tensor_{Z}\crp{F}L f$ it follows that
  $i\crp{F}Ki \iso \bigoplus_{b\in B} S \tensor b$
  as $\crp{F}P$\nbd module (remember that $L=\C_K(P)$).
  Thus $S\iso S\tensor b$ is a direct summand of $i\crp{F}Ki$ and
  thus a  permutation $P$\nbd module.

  Let $\beta=\br_P^{\crp{F}K}\colon(\crp{F}K)^P\to \crp{F}L$.
  Note that $\beta(i)=\beta(ef)=f^2=f$ and that
  $\beta(S^P)\subseteq \Z(\crp{F}Lf)$, as $L$ centralizes $S$.
  It follows
  $S^P/(S^P\cap\ker \beta) \iso \beta(S^P)
                           = \Z(\crp{F}L f)
                           \iso \crp{E}$.
  It remains to show that
  $S^P\cap \ker \beta   =\ker \br_P^S$.
  It is clear that
  \[\ker \br_P^S
     =\sum_{Q<P}S_Q^P
     \subseteq \sum_{Q<P} (\crp{F}K)_Q^P
     = \ker \beta.
  \]
  Conversely, suppose
  \[ s=\sum_{Q<P} \T_Q^P(a_Q) \in S^P \cap \ker \beta
     \quad \text{with $a_Q\in (\crp{F}K)^Q$.}\]
  Since $s=isi= \sum_{Q<P} \T_Q^P(ia_Q i)$,
  we may assume $a_Q\in (i\crp{F}K i)^Q$.
  Now $S$ is a direct summand of $i\crp{F}K i$ as
  $P$-module. Letting $\pi\colon i\crp{F}K i \to S$ be the
  corresponding projection, we see
  \[ s = \pi(s) = \sum_{Q<P} \T_Q^P(\pi(a_Q)) \in \sum_{Q<P} S_Q^P
       = \ker \br_P^S\]
  as was to be shown.
  The proof is finished.
\end{proof}
\begin{proof}[Proof of Theorem~\ref{t:glaubmag}]
  Suppose we have a counterexample to the theorem with
  $\abs{K/L}$ minimal.
  Then clearly $L<K$.
  Set $C=\C_P(K)$, so that $C<P$.
  Let $P_0/C$ be a chief factor of $H/C$ with $P_0\leq P$
  and set $L_0=\C_K(P_0)$.
  As $P_0>C$, we have $L_0<K$.
  Now the composition
  \[ (\crp{E}K)^P \xrightarrow{\br_{P_0}} (\crp{E}L_0)^{P/P_0}
                \xrightarrow{\br_{P/P_0}} \crp{E}L
  \]
  gives just $\beta=\br_P$.
  Let $\eta\in \Irr L_0$ be defined by
  $\br_{P_0}(\overline{e_{\theta}})=\overline{e_{\eta}}$.
  Set $M_0= KP_0$ and $H_0=\N_G(P_0)$.
  Note that $M_0\nteq G$.
  We have decomposed the configuration
  $(G,H,K,L,\theta,\phi)$ into two configurations:
  The configuration
  $(G,H_0, K, L_0, \theta, \eta)$ with $M_0$ and $P_0$
  instead of $M$ and $P$, and the
  configuration
  $(H_0, H, L_0, L, \eta, \phi)$ with $PL_0$ and $P$ instead of $M$
  and $P$ (see Figure~\ref{fig:glaubdecomp}).
  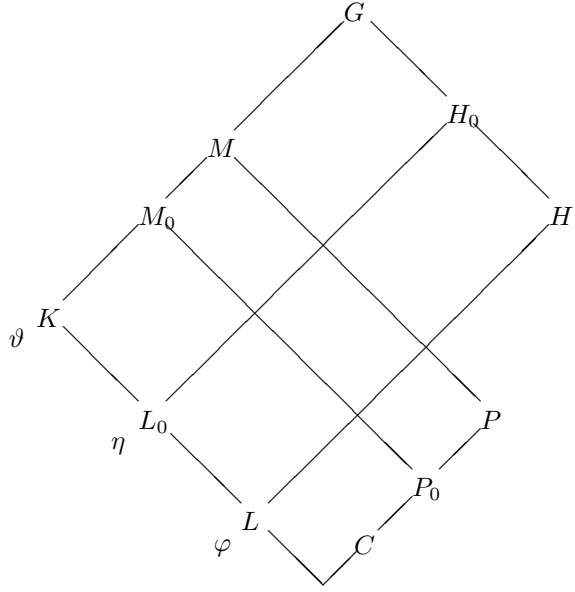
\begin{figure}[ht]
  \centering
  \setlength{\unitlength}{0.9mm}
  \begin{picture}(100,85)(-8,-1)
    \put(2,42){\line(1,1){11}}  
    \put(17,57){\line(1,1){6}}  
    \put(27,67){\line(1,1){16}} 
    \put(17,27){\line(1,1){41}}
    \put(32,12){\line(1,1){41}}
    \put(40,0){\line(1,1){5}}  
    \put(48,8){\line(1,1){5}}    
    \put(57,17){\line(1,1){6}}   
    \put(2,38){\line(1,-1){11}}  
    \put(17.7,22.3){\line(1,-1){10.3}}  
    \put(32,8){\line(1,-1){8}}    
    \put(17,53){\line(1,-1){36}}
    \put(27,63){\line(1,-1){36}}
    \put(47,83){\line(1,-1){11}} 
    \put(62,68){\line(1,-1){11}}  
    \put(-2,38){$K$}
    \put(-6,35){$\theta$}
    \put(13,23){$L_0$}
    \put(9,20){$\eta$}
    \put(28,8){$L$}
    \put(24,5){$\phi$}
    \put(44.5,4.5){$C$}
    \put(13,53){$M_0$}
    \put(53,13){$P_0$}
    \put(23,63){$M$}
    \put(63,23){$P$}
    \put(43,83){$G$}
    \put(58,68){$H_0$}
    \put(73,53){$H$}
  \end{picture}
  \caption{Proof of Theorem~\ref{t:glaubmag}}\label{fig:glaubdecomp}
  \end{figure}

  If $L<L_0$, then induction applies and yields magic crossed
  representations $\sigma_1$ and $\sigma_2$
  for the two configurations
  $(G,H_0, K, L_0, \theta, \eta)$ and
  $(H_0, H, L_0, L, \eta, \phi)$.
  Now note that all the hypotheses made in Section~\ref{sec:induc} are
  met.
  We apply Proposition~\ref{p:induc} to conclude that there is a
  magic crossed representation
  $\sigma\colon H\to S$, such that
  \[ \sigma(h)j = \sigma_1(h)\sigma_2(h)
     \quad
     \text{with}
     \quad
     j   = \sum_{\gamma}(\overline{e_{\theta}e_{\eta}e_{\phi}})^{\gamma}.
  \]
  By induction, we may also assume that
  $\br_{P_0}(\sigma_1(h))= \overline{e_{(\eta, \rats_p)}}$
  for $h \in \C_G(P_0)$
  and
  $\br_{P}(\sigma_2(h))= \overline{e_{(\phi, \rats_p)}}=f$ for
  $h\in \C_G(P)$.
  As $\br_P(\overline{e_{\theta}})=\br_P(\overline{e_{\eta}})=
  \overline{e_{\phi}}$, it follows that
  $\br_P(j)= f$.
  Thus for $h\in \C_G(P)$,
  \begin{align*}
    \br_P(\sigma(h))
      = \br_P( \sigma(h)j )
      &= \br_P( \br_{P_0}( \sigma(h) j) )
      \\
      &= \br_P( \br_{P_0}( \sigma_1(h)\sigma_2(h) ) )
      \\
      &= \br_P( \overline{e_{(\eta, \rats_p)}} \sigma_2(h))
      \\
      &= f.
  \end{align*}
  It follows that our configuration is not a counterexample.

  Thus we can assume $L=L_0$. We may replace $P$ by $P_0$ and assume
  that $P/C$ is a chief factor of $H$.
  In particular, we can assume that $P/C$ is  abelian.

  By Lemma~\ref{l:permp_basics}, Part~\ref{i:simplepalg},
  there is a unique $\sigma\colon P/C\to S$ which induces the action
  of $P$ on $S$.
  Let $V$ be a simple $S$\nbd module, which becomes an
  endopermutation $\crp{E}P$\nbd module via $\sigma$.
  By Lemma~\ref{l:turull33}, there is
  an $\crp{F}P$\nbd module $V_0$ such that
  $V\iso V_0\tensor_{\crp{F}}\crp{E}$.
  We may assume $V_0\leq V$.
  Let $S_0=\{ s\in S \mid V_0 s\subseteq V_0\}$.
  Then $\sigma(P)\subseteq S_0$ and $S_0\iso \mat_n(\crp{F})$.

  The action of $H/L$ on $S$ yields a projective crossed
  representation
  $\widetilde{\sigma}\colon H/L\to S^*$, that is, for each
  $h\in  H$ there is an element
  $\widetilde{\sigma}(h)=\widetilde{\sigma}(Lh)\in S^{*}$ such
  that
  $s^h = s^{\widetilde{\sigma}(h)}$ for all $s\in S_0$,
  that is, $\widetilde{\sigma}$ is crossed with respect to
  $S_0$.
  (Note that we may assume $\widetilde{\sigma}(x)=\sigma(x)$ for
  $x\in P$.)

  Let $x\in P$ and $h\in H$. By the uniqueness of $\sigma$ on $P$,
  we must have $\sigma(x)^h = \sigma(x^h)$.
  It follows that
  \[ \sigma(x)^{\widetilde{\sigma}(h)}
     = \sigma(x)^h
     = \sigma(x^h)
     \in \sigma(P).\]
  Thus $\widetilde{\sigma}(h)\in \N_{S^*}(\sigma(P))$.

  Let $\eps\colon H/L \to \Aut S$ be defined by
  $(s_0 z)^{\eps(h)}= s_0 \cdot z^h$ for
  $s_0\in S_0$ and $z\in \Z(S)\iso\crp{E}$.
  To $\widetilde{\sigma}$ belongs a cocycle
  $\alpha\in Z^2(H/L, \Z(S)^*)$ such that
  \[ \widetilde{\sigma}(x)^{\eps(y)}\widetilde{\sigma}(y)
      = \alpha(x,y)\widetilde{\sigma}(xy)
      \quad \text{for all $x$, $y\in H/L$}.
  \]
    Now we apply the homomorphism
    $\phi\colon \N_{S^*}(\sigma(P))\to S(P)^*\iso\crp{E}^*$
  of Lemmas~\ref{l:endoextsplit}
  and~\ref{l:endoextainv}
  to this equation.
  It follows that
  \[ \phi(\widetilde{\sigma}(x))^{\eps(y)}
     \phi(\widetilde{\sigma}(y))
     = \alpha(x,y)
       \phi(\widetilde{\sigma}(xy)).\]
  (Here,
  $\phi(\widetilde{\sigma}(x)^{\eps(y)})
   = \phi(\widetilde{\sigma}(h))^{\eps(y)}$
  follows from Lemma~\ref{l:endoextainv}.)
  Thus
  \[ \sigma(h):= \phi(\widetilde{\sigma}(h))^{-1}
                 \widetilde{\sigma}(h)\]
  defines a magic crossed representation
  $\sigma\colon H/L \to S$.
  Moreover, for this choice of
  $\sigma$ we have
  $\phi(\sigma(h))= 1_{S(P)}$.
  Here, we may identify $S(P)$ with
  $\Z(\crp{F}L f)$.
  Since $\phi$ extends
  $\br_P$, it follows that for
  $c\in \C_G(P)$ we have
  $\br_P(\sigma(c))=\phi(\sigma(c))=f$.
  The proof is finished.
\end{proof}
\printbibliography
\end{document}